\documentclass[11pt,a4paper]{article}
\usepackage[utf8]{inputenc}
\usepackage{amsmath,amssymb,amsthm,mathtools}
\usepackage[margin=28mm]{geometry}
\usepackage[hidelinks]{hyperref}
\usepackage{microtype}
\hypersetup{pdftitle={Finite Length of Locally Analytic Vectors for GL2(F) and D(F) Units},pdfauthor={Pengcheng Ding}}
\theoremstyle{plain}
\newtheorem{theorem}{Theorem}[section]
\newtheorem{lemma}[theorem]{Lemma}
\newtheorem{proposition}[theorem]{Proposition}
\newtheorem{corollary}[theorem]{Corollary}
\theoremstyle{definition}

\theoremstyle{remark}

\numberwithin{equation}{section}
\newcommand{\Q}{\mathbf Q}
\newcommand{\Z}{\mathbf Z}

\newcommand{\C}{\mathbf C}
\newcommand{\Qp}{\mathbf Q_p}

\newcommand{\GL}{\mathrm{GL}}
\newcommand{\SL}{\mathrm{SL}}
\newcommand{\an}{\mathrm{an}}
\newcommand{\gr}{\operatorname{gr}}

\newcommand{\vp}{v_p}
\newcommand{\bSH}{\overline S_H}
\newcommand{\bST}{\overline S_T}
\newcommand{\eps}{\varepsilon}
\begin{document}

\title{Finite Length of Locally Analytic Vectors\\
for $\mathrm{GL}_2(F)$ and $D_F^\times$}
\author{Pengcheng Ding}
\date{}
\maketitle

\begin{abstract}
Let $F/\mathbf Q_p$ and $E/\mathbf Q_p$ be finite extensions, and let
$G=\mathrm{GL}_2(F)$ or $D_F^\times$, where $D_F$ is the quaternion
division algebra over $F$. Let $\Pi$ be an admissible Banach
$E$-representation of $G$ whose locally $\mathbf Q_p$-analytic vectors
$\Pi^{\mathrm{an}}$ have an infinitesimal character.
If $p>2$ or $F=\mathbf Q_p$, we prove that $\Pi^{\mathrm{an}}$ and $\Pi$
have finite length in their respective admissible categories.
\end{abstract}

\section{Introduction}
\label{sec:introduction}
\subsection{The finite-length problem}
Let $F/\mathbf Q_p$ and $E/\mathbf Q_p$ be finite, let $D_F$ be the
quaternion division algebra over $F$, and put $G=D_F^\times$ or
$\mathrm{GL}_2(F)$. We study an admissible Banach $E$-representation
$\Pi$ of $G$ and its locally $\mathbf Q_p$-analytic vectors
$\Pi^{\mathrm{an}}$. Write
$\mathfrak g_{G,E}=E\otimes_{\mathbf Q_p}\operatorname{Lie}_{\mathbf Q_p}(G)$.
An infinitesimal character is an $E$-algebra homomorphism
\[
 \chi_\Pi:Z\bigl(U(\mathfrak g_{G,E})\bigr)\longrightarrow E
 \quad\text{such that}\quad
 zv=\chi_\Pi(z)v
 \quad(z\in Z(U(\mathfrak g_{G,E})),\ v\in\Pi^{\mathrm{an}}).
\]
Assuming that $\Pi^{\mathrm{an}}$ has an infinitesimal character, we ask
whether it has finite length in the admissible locally analytic category
and whether $\Pi$ has finite length in the category of admissible Banach
representations. Subobjects in the latter
category are closed $G$-stable subspaces.

These questions are motivated by the $p$-adic Jacquet--Langlands
correspondence. Scholze \cite{bib8} uses the Lubin--Tate tower to construct
cohomological functors from admissible smooth
$\GL_2(\Qp)$-representations with torsion coefficients to admissible
smooth $D_{\Qp}^\times$-representations with a commuting Galois action.
Pa\v{s}k\=unas extends this construction to admissible unitary Banach
representations, obtaining the functors $\check S^i$
\cite[\S3.3 and Lem. 3.4]{bib17}.
For $p>2$ and an absolutely irreducible nonordinary admissible unitary
Banach representation $\Pi$ of $\GL_2(\Qp)$ with a central character,
Dospinescu--Pa\v{s}k\=unas--Schraen show that $\check S^1(\Pi)^{\an}$
has the same infinitesimal character as $\Pi^{\an}$
\cite[Cor. 8.16]{bib2}.

For $G=\GL_2(\Qp)$, if $\Pi$ is unitary, has a central character, and
$\Pi^{\an}$ has an infinitesimal character, then $\Pi$ has finite length
\cite[Cor. 6.4]{bib2}, and so does $\Pi^{\an}$ \cite[Rem. 0.5(ii)]{bib11}. When $p>2$ and $\Pi$ is absolutely
irreducible and nonordinary, Dospinescu computes the infinitesimal
character of $\Pi^{\an}$ in terms of the Sen weights of its associated Galois
representation \cite[Thms. 1.1--1.2]{bib18}.

For $G=D_{\Qp}^\times$, \cite[Thm. 6.1(b)]{bib2} bounds the canonical
(Gelfand--Kirillov) dimension of $\Pi$ by one when $\Pi^{\an}$ has an
infinitesimal character. We bound the number of successive quotients of
canonical dimension one independently of the radius. The key remaining issue for finite
length is to control subquotients of $\Pi^{\an}$ of
canonical dimension zero. These are
locally algebraic on the derived group (Lemma~\ref{lem2-5}), so a
nonalgebraic infinitesimal character for the derived Lie algebra
excludes them. When $\Pi$ also has a central character, this yields
finite length of $\Pi$ \cite[Cor. 6.7(a)]{bib2}.

In the algebraic case and under the central character hypothesis,
\cite[Cor. 6.7(b)]{bib2} gives finite length of the distribution modules
obtained from the strong dual of $\Pi^{\an}$ at each fixed analytic
radius. Passing to $\Pi^{\an}$ then requires control of the
subquotients of canonical dimension zero uniformly as the radius tends to one.
For $p>2$, Pa\v{s}k\=unas proves that an admissible unitary Banach
representation of $D_{\Qp}^\times$ of canonical dimension zero is finite-dimensional,
hence has finite length \cite[Lem. 6.9]{bib17}. Under the global and
residual hypotheses of \cite[Prop. 6.15]{bib17}, the subspace of
$\check S^1(\Pi)$ locally algebraic on the reduced-norm-one subgroup is
finite-dimensional. Every Banach subrepresentation of canonical dimension
zero is contained in this finite-dimensional subspace
\cite[Prop. 6.13]{bib17}. Building on Pa\v{s}k\=unas's work,
Hu--Wang \cite{bib3} and Liu--Wang \cite{bib5} prove that further classes
of $D_{\Qp}^\times$-representations obtained from $\check S^1$ have finite
length as Banach representations.

\subsection{Main result and proof strategy}
\begin{theorem}\label{thm:finite-length}
Let $F/\mathbf Q_p$ be finite, let $G=D_F^\times$ or
$\mathrm{GL}_2(F)$, let $E/\mathbf Q_p$ be finite, and let $\Pi$ be an
admissible Banach $E$-representation of $G$. Suppose
$\Pi^{\mathrm{an}}$ has infinitesimal character $\chi_\Pi$ and assume
$p>2$ or $F=\mathbf Q_p$. Then $\Pi^{\mathrm{an}}$ has finite length
in the category of admissible locally analytic representations of $G$,
and $\Pi$ has finite length in the category of admissible Banach
representations of $G$.
\end{theorem}

We outline the proof for $p>2$. The aim is to bound the lengths of
finite $G$-stable coadmissible chains independently of the analytic
radius. Section~\ref{sec:distribution-algebras} develops the filtered
algebras and lattice comparisons that allow information from the
infinitesimal action to pass to associated graded modules.

In Section~\ref{sec:analytic-torus}, we follow Sorgdrager's method
\cite{bib9}. We choose a suitably weighted compact open subgroup $H$ of the
derived group $G_1$ and an integral lattice in the Banach dual.
Casimir elements and the infinitesimal character then give annihilation
relations on the associated graded module of the reduction of this lattice. We choose an analytic torus $T\subseteq H$ so that this
graded module is finite over the polynomial algebra arising from $T$.
Lifting graded generators gives finiteness over the torus distribution
algebra at each sufficiently large radius, with a uniform generator
bound (Proposition~\ref{prop3-3}). Corollary~\ref{cor3-5} identifies the
$H$- and $T$-radius modules of coadmissible subquotients.

Section~\ref{sec4} uses this finiteness to bound the number of successive
quotients of positive canonical dimension. At a fixed radius,
Theorem~\ref{thm4-3} and Lemma~\ref{lem4-9} give a uniform bound for
the ranks of subquotients of the completed analytic dual along their
support components. Propositions~\ref{prop4-5} and~\ref{prop4-8}
describe the support components of positive dimension and prove their
geometric irreducibility; their defining equations belong to a finite
family independent of the radius. Additivity of rank along these
components bounds the number of successive quotients of maximal
canonical dimension. Induction on canonical dimension then bounds the
number of successive quotients with positive canonical dimension,
uniformly in the radius
(Theorem~\ref{thm4-10}).

Section~\ref{sec5} treats the subquotients of canonical dimension zero. For a nonzero
subquotient in this dimension, Lemma~\ref{lem2-5} writes its strong dual
as a tensor product of a fixed algebraic representation and a smooth
admissible representation. The coadjoint orbit method of Boyarchenko--Sabitova \cite{bib1}
gives lower bounds for the dimensions of smooth irreducible
representations in terms of their conductors
(Proposition~\ref{prop5-2}). Galois conjugation over the fixed coefficient
field and the $G$-action yield the further estimates in
Propositions~\ref{prop5-3} and~\ref{prop5-4}.
Comparing these lower bounds with the upper bound of
Corollary~\ref{cor4-4} controls the smooth factors and gives uniform
bounds for the lengths of $G$-stable chains in canonical dimension zero
(Propositions~\ref{prop5-5} and~\ref{prop5-6}).

For each finite chain, we choose a radius at which the canonical
dimensions of all its nonzero successive quotients have stabilized.
Combining the bounds in positive canonical dimension and canonical
dimension zero then proves finite
length (Subsection~\ref{sec5-5}).

For $F=\Qp$, Appendix~\ref{appA} gives a simpler proof. The main
simplification is a direct upper bound for the dimensions of
subquotients of canonical dimension zero at a fixed radius in terms of
the conductors of their smooth factors, obtained from associated graded
modules. Ranks over a polynomial ring in one variable bound the number
of successive quotients of canonical dimension one. Combining these
estimates with Subsection~\ref{sec5-4} proves finite length for every
prime.

\subsection*{Notations}
We use a sufficiently large finite extension $E/\Qp$ as the coefficient
field. Let $F_0$ be the maximal unramified subextension of $F/\Qp$, and
put $d=[F:\Qp]$, $e=e(F/\Qp)$, and $f=[F_0:\Qp]$, so $d=ef$.
Write $\mathcal O_F$ and $\mathcal O_E$ for the rings of integers, fix
uniformizers $\pi_F$ and $\varpi$, and put $k_E=\mathcal O_E/(\varpi)$.
Normalize $v_p(p)=1$ and $v_F(\pi_F)=1$, and set $|a|=p^{-v_p(a)}$.

Write $G_1=\SL_1(D_F)$ for $G=D_F^\times$ and $G_1=\SL_2(F)$ for
$G=\GL_2(F)$, and put
$\mathfrak g_E=E\otimes_{\Qp}\operatorname{Lie}_{\Qp}(G_1)$.
Locally analytic means locally $\Qp$-analytic; group dimensions are over
$\Qp$.

Write $(-)'$ for the continuous dual, $(-)'_b$ for the strong dual, and
$(-)^\vee$ for the contragredient dual of a finite-dimensional representation.
Dual actions are given by $(g\lambda)(v)=\lambda(g^{-1}v)$.

We write $c=c(A,B,\ldots)$ to mean that $c$ depends only on the
indicated objects and on choices fixed throughout the paper.

\subsection*{Acknowledgements}
The author thanks his advisor, Yongquan Hu, for helpful discussions and
support.

\subsection*{Author contributions and use of AI tools}
The author formulated the initial conjecture over $\Q_p$ and established
the upper bound on dimension in terms of the conductor in
Proposition~\ref{propA-3}. GPT-5.6 Sol (OpenAI) developed the arguments
giving the lower bounds in Propositions~\ref{prop5-2} and~\ref{prop5-3} for the case
$F=\Q_p$. Combining these estimates yielded the proof over $\Q_p$.
GPT-6 Astra (OpenAI) subsequently developed the extension to the general
case. These tools also assisted with calculations and the exposition
of the manuscript.

The author revised the overall proof strategy, including the organization
of the arguments at each fixed radius and the passage to the global
conclusion, and improved several proofs to make their reasoning more
transparent. An earlier proof of Proposition~\ref{prop4-8} relied on a
strong geometric result; following discussions and suggestions by the
author, an algebraic proof was developed to replace that argument.
The author made the final decisions on the arguments and results included
in the manuscript, checked all proofs, and takes full responsibility for
its mathematical correctness and final content.

\section{Distribution algebras and locally analytic vectors}
\label{sec:distribution-algebras}
We take $E$ to contain the normal closure of $F$ and to split $D_F$,
and identify $F$ with a subfield of $E$. We also assume
$p^{-w}\in|E^\times|$ for the finitely many rational basis weights used
below. These conditions and the additional ones in
Section~\ref{sec:analytic-torus} require only a fixed finite coefficient
extension, independent of the radius. Finite length descends to the
original field, since finite scalar extension is exact and preserves
nonzero quotients.

\subsection{Groups and valuations}
\label{subsec:groups-valuations}
Let $H$ be a compact $p$-adic analytic group of dimension $d_H\geq1$, equipped with a strictly saturated $p$-valuation $\omega_H$. Choose an adapted ordered basis $g_1,\ldots,g_{d_H}$. Thus the ordered coordinate map $\Z_p^{d_H}\to H$ is a homeomorphism, and
\[
 \omega_H(g_1^{a_1}\cdots g_{d_H}^{a_{d_H}})
 =\min_j\{\vp(a_j)+w_j\},\qquad w_j=\omega_H(g_j).
\]
Saturation means that every element of valuation greater than $p/(p-1)$ is a $p$th power in $H$. Strict saturation additionally requires
\begin{equation}\label{eq2-1}
 \frac1{p-1}<w_i<\frac p{p-1}\qquad(1\leq i\leq d_H)
\end{equation}
\cite[Def. 4.1]{bib9}. Put
\[
 \mathfrak h=\log H,\qquad x_j=\log g_j,\qquad
 \mathfrak h_E=\mathfrak h\otimes_{\Z_p}E,\qquad U=U(\mathfrak h_E).
\]
Lazard's logarithm identifies $H$ with its Lie lattice $\mathfrak h$ \cite{bib4}. For $n\geq0$, the subgroup
\[
 H_n=H^{p^n}=\exp(p^n\mathfrak h)
\]
is normal in $H$, with ordered basis $g_1^{p^n},\ldots,g_{d_H}^{p^n}$. Its induced strictly saturated valuation is $\omega_{H,n}=\omega_H|_{H_n}-n$, so these basis elements again have weights $w_i$ \cite[\S4]{bib9}.

\subsection{Distribution algebras and completions}
\label{subsec:distribution-completions}
The completed group algebra is
\[
 \Lambda_H^\circ=\mathcal O_E[\![H]\!]
 =\varprojlim_{K'\triangleleft_o H}\mathcal O_E[H/K'],
\]
where $K'\triangleleft_o H$ means that $K'$ is an open normal subgroup
of $H$. Put $\Lambda_H=E[\![H]\!]:=\Lambda_H^\circ[1/p]$ and
$\overline\Lambda_H=\Lambda_H^\circ/\varpi\Lambda_H^\circ=k_E[\![H]\!]$.
Write $D_H=D(H,E)=C^{\an}(H,E)'_b$ for the locally analytic distribution
algebra. We use the same notation for a group element and its images in
these algebras; the image of $g\in H$ in $D_H$ is the Dirac distribution
$f\mapsto f(g)$. Put $b_j=g_j-1$ and use ordered monomials $b^\alpha$ and
$x^\alpha$. The Lie embedding is given by
\[
 x\longmapsto\left[f\longmapsto
  \left.\frac d{dt}f(\exp(tx))\right|_{t=0}\right]
 \qquad(x\in\mathfrak h\otimes_{\Z_p}\Qp)
\]
and extends $E$-linearly to $\mathfrak h_E$ \cite[p.~449]{bib7}. It identifies $x_j$ with $\log(g_j)$ in $D_H$. On the contragredient dual, the
Lie algebra action is $(x\lambda)(v)=-\lambda(xv)$.

For $1/p\leq r<1$, write $D_r=D_r(H,E;\omega_H)$ for the completion of $D_H$ with respect to the norm
\[
 \left\|\sum_\alpha a_\alpha b^\alpha\right\|_r
 =\sup_\alpha|a_\alpha|r^{\sum_iw_i\alpha_i}.
\]
Its elements have ordered expansions with $|a_\alpha|r^{\sum_iw_i\alpha_i}\to0$.
Set $r_n=p^{-1/p^n}$ for $n\geq0$, and put
\[
 D_{H,n}=D_{r_n}(H,E;\omega_H),\qquad
 \widehat U_n=D_{1/p}(H_n,E;\omega_{H,n}).
\]
The index $n$ in $D_{H,n}$ specifies the radius $r_n$ for the fixed group $H$. We will identify $\widehat U_n$ with a completed enveloping algebra in Proposition~\ref{prop2-1} below. The norm $\|\cdot\|_r$ defines the following filtration:
\[
\begin{gathered}
 \operatorname{Fil}_r^sD_r=\{\lambda:\|\lambda\|_r\leq p^{-s}\},\qquad
 \gr_r^sD_r=\operatorname{Fil}_r^sD_r/\operatorname{Fil}_r^{>s}D_r,\\
 \operatorname{Fil}_r^{>s}=\bigcup_{t>s}\operatorname{Fil}_r^t.
\end{gathered}
\]
Write $\operatorname{gr}$ for associated graded rings and modules with
the specified filtration.
For a filtered module $M$, we use the shift convention $\operatorname{Fil}^s(M(-d))=\operatorname{Fil}^{s-d}M$; for a graded module, $M(-d)_s=M_{s-d}$. Recall that a homomorphism $f:M\to N$ preserving the filtrations is strict if
\[
 f(\operatorname{Fil}^sM)=f(M)\cap\operatorname{Fil}^sN
 \qquad\text{for every }s.
\]

For $0\ne\lambda\in D_r$, its principal symbol is
\[
 \sigma_r(\lambda)=\lambda+\operatorname{Fil}_r^{>s}D_r\in\gr_r^sD_r,
 \qquad s=-\log_p\|\lambda\|_r.
\]
Set $\sigma_r(0)=0$. For coefficients, let $\sigma:E\to\gr E$ denote the principal symbol map
associated with $\vp$. Ordered expansions show that $D_H\to D_r$ is injective. Hence the Lie embedding identifies $U$ with a subalgebra of every $D_r$.

For $n\geq0$, let $\widehat U_{\omega_{H,n}}(p^n\mathfrak h)_E$ denote
the completion of $U$ for the weighted Poincar\'{e}--Birkhoff--Witt (PBW) norm
\[
 \left\|\sum_\alpha a_\alpha(p^nx)^\alpha\right\|_{\omega_{H,n}}
 =\sup_\alpha|a_\alpha|p^{-\sum_iw_i\alpha_i}.
\]
Choose representatives $g$ for $H_n\backslash H$. The relative crossed
product $\widehat U_n\rtimes_{H_n}H$ is $\bigoplus_g\widehat U_n g$ as a
left $\widehat U_n$-module. Its multiplication extends that of
$\widehat U_n$ and the group law of $H$, with
\[
 gbg^{-1}=\operatorname{Ad}(g)(b),\qquad h=\exp(\log h)
\]
for $g\in H$, $b\in\widehat U_n$, and $h\in H_n$, with the exponential
taken in $\widehat U_n$. Here $\operatorname{Ad}(g)$ is the automorphism
of $\widehat U_n$ induced by conjugation on $H_n$.

The following descriptions for the weighted completions follow from
\cite[Prop. 3.1, Lem. 4.1, and (11)]{bib9}; for uniform pro-$p$ groups,
see \cite[Cor. 6.5.4 and Thm. 6.5.11]{bib10}.

\begin{proposition}\label{prop2-1}
For every $n\geq0$, the Lie and group embeddings give topological
$E$-algebra isomorphisms
\begin{equation}\label{eq2-2}
 \widehat U_{\omega_{H,n}}(p^n\mathfrak h)_E\simeq\widehat U_n,
 \qquad D_{H,n}\simeq\widehat U_n\rtimes_{H_n}H.
\end{equation}
In particular, $\widehat U_n$ is the closure of $U$ in $D_{H,n}$.
\end{proposition}
\begin{proof}
By \cite[Prop. 3.1]{bib9}, the Lie embedding $U\to\widehat U_n$ extends isometrically to the completion. Since $w_i>1/(p-1)$, the exponential series converges and the image contains all ordered polynomials in $g_i^{p^n}-1$. The image of the isometric extension is closed and contains a dense subalgebra of $\widehat U_n$, so the map is surjective. The second isomorphism follows from \cite[Lem. 4.1 and (11)]{bib9}, with multiplication determined by conjugation and the embedding $H_n\to\widehat U_n^\times$.
\end{proof}
The identity on $U$ gives continuous inclusions $\widehat U_n\hookrightarrow\widehat U_1$ for $n\geq1$: on the chosen Lie algebra generators, $p^nx_i=p^{n-1}(px_i)$.

For the filtered arguments below, assume that the $w_i$ are rational and that $H$ has been chosen as a $p$-power subgroup:
\[
 H=J^p,\qquad\omega_H=\eta|_H-1,
\]
where $(J,\eta)$ is strictly saturated and $g_j=\widetilde g_j^p$ for an adapted ordered basis $\widetilde g_1,\ldots,\widetilde g_{d_H}$ of $J$. Every strictly saturated group gives such an $H$; passing to $J^p$ preserves its ordered basis weights. For this shifted valuation,
\[
 \omega_H([g_i,g_j])\geq w_i+w_j+1>1+\frac{w_i+w_j}p.
\]
By \cite[Lem. 4.3]{bib9}, for every $1/p\leq r<1$ the ring $\gr_r\Lambda_H^\circ$ is a commutative polynomial algebra over $k_E$ in the symbols of $\varpi,b_1,\ldots,b_{d_H}$. Inverting the symbol of $\varpi$ gives $\gr_rD_r$.

We use the sequence of radii $(r_n)_{n\geq1}$. At each fixed radius $r_n$, a common rescaling makes all filtration degrees integral. The algebra $\Lambda_H$ is Noetherian and Auslander regular by \cite[Thm.
3.29]{bib19}, after a finite separable extension of the coefficient field as
in \cite[\S2.6]{bib20}. The same properties for $D_{H,n}$ follow from the criterion for filtered rings \cite[Thm. 3.4]{bib2}. The argument using the unit ball in \cite[proof of Thm. 8.9]{bib6} applies to these polynomial graded rings and gives $\operatorname{gldim}D_{H,n}\leq d_H$. Rationality of the weights allows the required finite coefficient extension at each radius; the bound then descends to $E$.

For $\Lambda_H^\circ\to D_{H,n}$, the associated graded map inverts the
coefficient symbol $\sigma(\varpi)$, as in \cite[proof of Prop. 4.7]{bib6}.
For the transition maps, after a finite coefficient extension at the two
radii, the intermediate completions in \cite[proofs of Lem. 4.8 and
Thm. 4.9]{bib6} reduce the graded maps to the inclusion of a polynomial
ring into a formal power series ring and localisation in the coefficient
symbol. The ordered expansions and polynomial graded rings above give
these same descriptions for our weights. These graded maps are flat. The flatness criterion for filtered rings and descent from the coefficient extensions therefore show that
$\Lambda_H\to D_{H,n}$ and $D_{H,n+1}\to D_{H,n}$ are left and right flat.
The transition maps $D_{H,n+1}\to D_{H,n}$ have dense image. As $r\to1$, the
weighted norms are cofinal with the norms $\sum_\alpha a_\alpha
b^\alpha\mapsto\sup_\alpha|a_\alpha|r^{|\alpha|}$. Hence
\[
 D_H=D(H,E)\simeq\varprojlim_{n\geq1}D_{H,n}
\]
with its strong dual topology, so $(D_{H,n})_{n\geq1}$ is a Fr\'{e}chet--Stein presentation of $D_H$.

\subsection{Coadmissible modules and integral lattices}
\label{subsec:coadmissible-lattices}
Let $M$ be a finite $\Lambda_H$-module, and put
\[
 M_n=D_{H,n}\otimes_{\Lambda_H}M.
\]
The $D_H$-module $D_H\otimes_{\Lambda_H}M$ is finitely presented,
hence coadmissible by \cite[Cor. 3.4(v)]{bib6}, and satisfies
\[
 D_H\otimes_{\Lambda_H}M\simeq\varprojlim_{n\geq1}M_n.
\]

For a coadmissible $D_H$-module $Q$, put
\[
 Q_n=D_{H,n}\otimes_{D_H}Q,\qquad Q\simeq\varprojlim_nQ_n,
 \qquad D_{H,n}\otimes_{D_{H,n+1}}Q_{n+1}\simeq Q_n.
\]
We call $Q_n$ the radius module of $Q$ at $r_n$. Each $Q_n$ is finitely generated over $D_{H,n}$, and $Q\mapsto Q_n$ is exact \cite[Rem. 3.2 and Cor. 3.4]{bib6}. If $Q\ne0$, then $Q_n\ne0$ for all sufficiently large $n$, by the transition identities.

A continuous Banach representation of the compact group $H$ admits an equivalent $H$-invariant norm. Admissibility of a Banach $H$-representation means finite generation of its continuous dual over $\Lambda_H$. For a Banach representation of a locally analytic group, admissibility is checked on compact open subgroups \cite[\S7]{bib6}. For an admissible Banach $H$-representation, the continuous integral dual of an invariant open bounded lattice is finitely generated over $\Lambda_H^\circ$. Locally analytic admissibility is used in the sense of Schneider--Teitelbaum. If $M=\Sigma'$ for an admissible Banach $H$-representation $\Sigma$, duality for locally analytic vectors gives
\begin{equation}\label{eq2-3}
 (\Sigma^{\an})'_b\simeq D_H\otimes_{\Lambda_H}\Sigma'\simeq\varprojlim_nM_n
\end{equation}
\cite[Thms. 5.2 and 7.1]{bib6}.

Fix a finite, $\varpi$-torsion-free $\Lambda_H^\circ$-lattice
$M^\circ\subset M$ and a finite free presentation with surjection
\[
 \mathcal P:(\Lambda_H^\circ)^a\twoheadrightarrow M^\circ.
\]
At each radius $r_n$, use the product filtration on the free module
and give $M^\circ$ its quotient filtration through $\mathcal P$.
Give $M_n$ the corresponding quotient filtration after base change to
$D_{H,n}$.
For submodules and quotients, use the induced submodule and quotient
filtrations, respectively.

The filtrations defined by $\|\cdot\|_{r_n}$ are discrete and
Zariskian, so the module filtrations above are good. The resulting short exact sequences remain exact
on associated graded modules. To compare with reduction modulo
$\varpi$, put $\overline M=M^\circ/\varpi M^\circ$ and equip it with
the quotient of the $\omega_H$-weighted filtration on $k_E[\![H]\!]$.
Write $\eps=\sigma(\varpi)$.
The graded module $\gr\overline M$ is independent of the radius. After rescaling the degrees on $\gr_{r_n}M^\circ/\eps\gr_{r_n}M^\circ$ to those of $\gr\overline M$, the comparison is
\begin{equation}\label{eq2-4}
\begin{gathered}
 \eps\text{ is regular on }\gr_{r_n}M^\circ,\qquad
 \frac{\gr_{r_n}M^\circ}{\eps\gr_{r_n}M^\circ}\simeq\gr\overline M,\\
 (\gr_{r_n}M^\circ)[1/\eps]\simeq\gr_{r_n}M_n
\end{gathered}
\end{equation}
for every $n\geq n_{\mathrm{res}}$, where $n_{\mathrm{res}}=n_{\mathrm{res}}(H,M)\geq1$ \cite[Prop. 2.1, Lems. 2.3--2.4, and Prop. 5.1]{bib9}. The graded rings are
\[
 \gr\overline\Lambda_H=k_E[X_1,\ldots,X_{d_H}],\qquad
 \gr_{r_n}\Lambda_H^\circ=k_E[\eps,X_{1,n},\ldots,X_{d_H,n}].
\]
The variable $X_{i,n}$ is the symbol of $b_i$ at radius $r_n$, with
\[
 \deg_{r_n}X_{i,n}=\frac{w_i}{p^n},\qquad\deg\eps=\vp(\varpi).
\]
Modulo $\eps$, multiply degrees by $p^n$ to identify $X_{i,n}$ with the variable $X_i$ of $\gr\overline\Lambda_H$, of degree $w_i$. These positive rational degrees may be scaled to integers. When using \eqref{eq2-4}, we suppress the radius index on the variables.

For a nonzero finite $R$-module $Y$, where $R=\Lambda_H$ or $D_{H,n}$, define its canonical dimension by
\[
 d_R(Y)=d_H-\inf\{i:\operatorname{Ext}^i_R(Y,R)\ne0\}.
\]
For an admissible Banach representation $\Pi$, $\operatorname{GKdim}\Pi$
is the canonical dimension of $\Pi'$ over $E[\![K]\!]$, for a compact open
uniform pro-$p$ subgroup $K\subset G$; it is independent of $K$
\cite[\S3.4]{bib2}.
For coadmissible modules put $d_{D_H}(Q)=\sup_n d_{D_{H,n}}(Q_n)$, with dimension $-\infty$ for the zero module. These dimension functions are exact, and the dimensions $d_{D_{H,n}}(Q_n)$ are nondecreasing and eventually equal to $d_{D_H}(Q)$ for $Q\ne0$ \cite[\S8, condition (DIM), Lem. 8.4, and Thm. 8.9]{bib6}.
For an admissible locally analytic $H$-representation $\Sigma$, its
canonical dimension is $d_{D_H}(\Sigma'_b)$.
Write $\operatorname{Kdim}$ for Krull dimension.
For a good filtration on $Y$ compatible with the filtration defined
by $\|\cdot\|_{r_n}$ on $D_{H,n}$,
\begin{equation}\label{eq2-5}
 d_{D_{H,n}}(Y)=\operatorname{Kdim}(\gr_{r_n}Y)-1
\end{equation}
\cite[Thm. 3.4]{bib2}. The coefficient ring is $\gr E=k_E[\eps,\eps^{-1}]$, where $\eps=\sigma(\varpi)$; its Krull dimension one accounts for the subtraction. Using \cite[Thm. 3.3]{bib20} in the integral form
\cite[Prop. 3.5]{bib2}, together with \eqref{eq2-5} and \eqref{eq2-4},
we obtain, for $M\ne0$,
\[
 \operatorname{Kdim}\gr\overline M=d_{\Lambda_H}(M)=d_{D_{H,n}}(M_n)
 \qquad(n\geq n_{\mathrm{res}}),
\]
after increasing $n_{\mathrm{res}}$ if necessary.

\begin{lemma}\label{lem2-2}
Let $Y$ be a nonzero finite $D_{H,n}$-module. Then
\[
 d_{D_{H,n}}(Y)=0\quad\Longleftrightarrow\quad\dim_EY<\infty.
\]
\end{lemma}
\begin{proof}
Suppose first that $d_{D_{H,n}}(Y)=0$. We use the lattice argument of \cite[Lem. 3.12 and Prop. 3.5]{bib2}. For fixed $n$, choose a finite extension $L/E$ with $r_n^{w_i}\in|L^\times|$ for all $i$. Extend $\|\cdot\|_{r_n}$ to $D_{H,n}\widehat\otimes_E L$. By the description of the associated graded rings above, the reduction of its unit ball modulo the maximal ideal of $\mathcal O_L$ is a polynomial algebra over the residue field of $L$. By \cite[Prop. 3.5]{bib2}, the reduction of a suitable lattice in $L\otimes_EY$ has Krull dimension zero over this polynomial algebra. A finite basis of the reduction over the residue field lifts, by $\varpi_L$-adic completeness, to a finite set of $\mathcal O_L$-generators of the lattice. Thus $L\otimes_EY$ is finite-dimensional over $L$, so $Y$ is
finite-dimensional over $E$.

Conversely, if $Y\ne0$ is finite-dimensional over $E$, then
$\gr_{r_n}Y$ is a nonzero finite graded module over $\gr E$.
Its Krull dimension is one, so \eqref{eq2-5} gives
$d_{D_{H,n}}(Y)=0$.
\end{proof}

\subsection{Restriction to the derived group}
\label{subsec:derived-algebraic}
\begin{lemma}\label{lem2-3}
Let $\Pi$ be an admissible Banach representation of $G$. Suppose either that $\Pi$ has a central character or that $\Pi^{\an}$ has an infinitesimal character. Then some compact open subgroup $Z_0\subseteq Z(G)$ acts on $\Pi$ through a locally analytic character. Moreover, the restriction $\Pi|_{G_1}$ is admissible, and the natural inclusion identifies $(\Pi|_{G_1})^{\an}$ with $\Pi^{\an}$ as locally convex spaces.
\end{lemma}
\begin{proof}
Put $M=\Pi'$. If $\Pi$ has a central character, that character is locally analytic on a sufficiently small open subgroup of $Z(G)$. If $\Pi^{\an}$ has an infinitesimal character, choose a sufficiently small $p$-valued compact open subgroup $K\subset G$. Admissibility gives a finite, $\varpi$-torsion-free $\mathcal O_E[\![K]\!]$-lattice in $M$. By \cite[Cor. 5.2]{bib9}, we can choose a radius $r$ for which the following
natural map is injective:
\[
 \iota:M\hookrightarrow Y_r:=D_r(K,E)\otimes_{E[\![K]\!]}M,
 \qquad m\longmapsto1\otimes m.
\]
The infinitesimal character on the analytic dual passes to $Y_r$ and gives the scalar action $z|_{Y_r}=-\chi_\Pi(z)\operatorname{id}_{Y_r}$ for $z\in\mathfrak z_E:=\operatorname{Lie}(Z(G))\otimes_{\Qp}E$. After fixing this radius, take $z$ in a sufficiently small central Lie lattice so that $\exp z\in K$ and its exponential series converges in $D_r(K,E)$. Then
\[
 \iota(\exp(z)m)=\exp(z)\iota(m)=\exp(-\chi_\Pi(z))\iota(m)
 \qquad(m\in M).
\]
Since $\iota$ is injective, $\exp z$ acts on $M$ by the scalar $\exp(-\chi_\Pi(z))$. By contragredient duality it acts on $\Pi$ by $\exp(\chi_\Pi(z))$, a locally analytic character of a sufficiently small central subgroup $Z_0$.

Choose a sufficiently small compact open $H\subseteq G_1$ and shrink $Z_0$ so that $H\times Z_0\to HZ_0$ is an isomorphism onto an open subgroup of $G$. This follows from $\operatorname{Lie}(G)=\operatorname{Lie}(G_1)\oplus\operatorname{Lie}(Z(G))$. Let $\psi_M:Z_0\to E^\times$ be the character through which $Z_0$ acts on $M$. As in \cite[proof of Cor. 6.7]{bib2}, this action factors through the continuous quotient map
\[
 E[\![H\times Z_0]\!]\longrightarrow E[\![H]\!],
 \qquad(h,z)\longmapsto\psi_M(z)h.
\]
Since $\Pi|_{HZ_0}$ is admissible, the quotient map makes $M$ finite
over $\Lambda_H$. Thus $\Pi|_{G_1}$ is admissible.

For the assertion about locally analytic vectors, recall that $v\in\Pi$ is locally analytic precisely when its orbit map $g\mapsto gv$ is locally analytic \cite[Thm. 3.5.7]{bib12}; this can be checked near the identity. Let $\psi:Z_0\to E^\times$ be the locally analytic character giving the action on $\Pi$. Under $H\times Z_0\simeq HZ_0$, the orbit map of $v$ becomes
\[
 (h,z)\longmapsto hzv=\psi(z)\,hv.
\]
Multiplication by $\psi(z)$ and restriction to $z=1$ show that this map is locally analytic exactly when $h\mapsto hv$ is. Since $HZ_0$ is open in $G$ and $H$ is open in $G_1$, \cite[Lem. 3.5.4]{bib12} gives the equality of the locally analytic vectors for $G$ and $G_1$. On sufficiently small analytic charts, multiplication by the fixed character $\psi$ and restriction to $z=1$ are continuous on the corresponding Banach spaces of analytic functions. Passing to the inductive limits that define the locally analytic vector topologies shows that this identification is topological.
\end{proof}

\subsection{Infinitesimal characters and subquotients of canonical dimension zero}
\label{subsec:conductors-radii}
With the coefficient field fixed above,
$\mathfrak g_E\simeq\bigoplus_{\sigma:F\hookrightarrow E}\mathfrak{sl}_2(E)$.

From now on, assume that $\Pi^{\an}$ has infinitesimal character $\chi_\Pi$.
Put $M=\Pi'$ and $M_{\an}=(\Pi^{\an})'_b$.
The infinitesimal character on $M_{\an}$ restricted to
$Z(U(\mathfrak g_E))$ is
\[
 \chi=(\chi_\Pi\circ\iota_U)\big|_{Z(U(\mathfrak g_E))},
\]
where $\iota_U$ is the principal anti-automorphism of
$U(\mathfrak g_{G,E})$, defined by $\iota_U(x)=-x$ for
$x\in\mathfrak g_{G,E}$.
Choose a compact open $H\subseteq G_1$ satisfying the conditions of
Subsection~\ref{subsec:distribution-completions}. Lemma~\ref{lem2-3} and
\eqref{eq2-3} give $M_{\an}=D_H\otimes_{\Lambda_H}M$.
For $\Gamma\in\{H,G_1,G\}$, use the action on $M_{\an}$ obtained by
duality from $\Pi^{\an}$. The quadratic Casimirs used below are fixed
by $\iota_U$.

Call $\chi$ algebraic if it is the character of a finite-dimensional algebraic representation of the derived group, viewed as a $\Qp$-algebraic group by restriction of scalars from $F$. Such a character determines a unique irreducible algebraic representation $V$, called the algebraic type; put $d_V=\dim_E V$. Since $\mathfrak g_E$ is a product of copies of $\mathfrak{sl}_2(E)$, $V$ is self-dual. We write $V^\vee$ for the algebraic factor in the dual module.

By \cite[Prop. 3.2]{bib21} and \cite[Lem. 6.5]{bib2}, every
finite-dimensional locally analytic $H$-representation with character
$\chi$ is semisimple. It is zero if $\chi$ is nonalgebraic; otherwise
it has the form $V^\vee\otimes_EW^\vee$, with $W$ smooth.
When $\chi$ is algebraic, shrink $H$ so that the matrices of a basis
of $\log H$ on $V$ and $V^\vee$ are sufficiently small.
Proposition~\ref{prop2-1} then gives continuous extensions of their
$H_1$-actions to $\widehat U_1$.

For a finite-dimensional smooth $H$-representation $W$, set
\[
 c_H(W)=\min\{n\geq0:H_n\subseteq\ker W\},\qquad c_H(0)=0.
\]

\begin{lemma}\label{lem2-4}
Assume that $\chi$ is algebraic, with algebraic type $V$ as above. Let $W$ be a finite-dimensional smooth $H$-representation, and put
\begin{equation}\label{eq2-6}
 S=V^\vee\otimes_E W^\vee,
\end{equation}
with its diagonal $H$-action. For every $n\geq1$, this action extends continuously to $D_{H,n}$ if and only if $c_H(W)\leq n$.
\end{lemma}
\begin{proof}
The case $W=0$ is immediate. Fix $n\geq1$ and first suppose that $c_H(W)\leq n$. Let $\widehat U_n$ act on $S$ through $V^\vee$, and let $H$ act diagonally. The $\widehat U_n$-action is continuous by the choice of $H$ and the inclusion $\widehat U_n\hookrightarrow\widehat U_1$. These actions agree on $H_n$, since $H_n$ acts trivially on $W^\vee$. The conjugation relation holds on the dense subalgebra $E[H_n]\subseteq\widehat U_n$, hence on $\widehat U_n$ by continuity. Proposition~\ref{prop2-1} therefore gives a continuous $D_{H,n}$-action.

Conversely, suppose that the diagonal action extends continuously to $D_{H,n}$. We adapt \cite[Lem. 6.6]{bib2} to the weighted completions using Proposition~\ref{prop2-1}. The resulting $\widehat U_n$-action agrees on $U$ with the action through $V^\vee$, since the smooth factor has zero Lie algebra action. Density and continuity imply agreement on $\widehat U_n$, and hence on its exponential subgroup $H_n$. Comparing with the diagonal action, and using $V\ne0$, shows that $H_n$ acts trivially on $W^\vee$. Thus $c_H(W)\leq n$.
\end{proof}

\begin{lemma}\label{lem2-5}
Let $\Gamma=H$ or $G_1$, and let $Q\ne0$ be a $\Gamma$-stable coadmissible subquotient of $M_{\an}$. Suppose $Q$ has canonical dimension zero. Then $\chi$ is algebraic and
\[
 Q'_b\simeq V\otimes_E\rho,\qquad
 Q_n\simeq V^\vee\otimes_E(\rho^{H_n})^\vee\qquad(n\geq1),
\]
where $\rho$ is a smooth admissible representation of $\Gamma$, equipped with its finest locally convex topology. The first isomorphism is topological and $\Gamma$-equivariant; the second is $D_{H,n}$-linear.
\end{lemma}
\begin{proof}
By Lemma~\ref{lem2-2}, every radius module $Q_n$ is finite-dimensional. By \cite[Prop. 6.5]{bib6}, $Q'_b\simeq\varinjlim_{n\geq1}Q_n^\vee$ with the finest locally convex topology. Some radius module $Q_n$ is nonzero, so \cite[Lem. 6.5]{bib2} shows that $\chi$ is algebraic. Applying the same lemma to each $Q_n^\vee$ shows that it is a direct sum of copies of $V$ as a $\mathfrak g_E$-module. Thus \cite[Props. 4.2.4 and 4.2.10]{bib12} give the topological $\Gamma$-equivariant evaluation isomorphism
\[
 V\otimes_E\rho\xrightarrow{\sim}Q'_b,\qquad
 \rho=\operatorname{Hom}_{\mathfrak g_E}(V,Q'_b),
\]
given by $v\otimes f\mapsto f(v)$. The action on $\rho$ is
\[
 (\gamma f)(v)=\gamma\bigl(f(\gamma^{-1}v)\bigr)
 \qquad(\gamma\in\Gamma,\ f\in\rho,\ v\in V).
\]
Since $V$ is finite-dimensional, each $f\in\rho$ has image in some
$Q_n^\vee$. The algebraic--smooth decomposition of $Q_n^\vee$ makes
$f$ fixed by an open subgroup of $H$. Thus $\rho$, with this action
and its finest locally convex topology, is a smooth $\Gamma$-representation.

The canonical map $Q\to Q_n$ has dense image, so its dual identifies $Q_n^\vee$ with a subspace of $Q'_b$. Put $W_n=\operatorname{Hom}_{\mathfrak g_E}(V,Q_n^\vee)\subseteq\rho$. Lemma~\ref{lem2-4} gives $W_n\subseteq\rho^{H_n}$. Conversely, for $w\in\rho^{H_n}$, put $W=E[H/H_n]w$. Duality gives a continuous quotient $f:Q\twoheadrightarrow V^\vee\otimes_E W^\vee$. Lemma~\ref{lem2-4} extends the action on its target to $D_{H,n}$, so the tensor-product universal property gives a surjection
\[
 Q_n=D_{H,n}\otimes_{D_H}Q\longrightarrow V^\vee\otimes_EW^\vee,
 \qquad a\otimes q\longmapsto a\,f(q).
\]
Dualizing gives $W\subseteq W_n$, hence $W_n=\rho^{H_n}$. This proves the formula for $Q_n$ and admissibility of $\rho$.
\end{proof}

The smooth factor $\rho$ in Lemma~\ref{lem2-5} may be infinite-dimensional,
even though $Q$ has canonical dimension zero.

\section{Restriction to an analytic torus}
\label{sec:analytic-torus}
Let $p>2$. We also assume that $E$ contains the unramified quadratic
extension of $F_0$ and square roots of all conjugates of $\pi_F$.

We follow \cite[\S\S6.2--6.7]{bib9}, writing
$F,E,\pi_F,\vartheta$ for $K,L,\varpi_K,\Pi$ there.
We use trace-zero lattices scaled by a sufficiently large power of $p$
and the weights specified in \eqref{eq3-4} below.

\subsection{The weighted subgroup and its torus}
\label{subsec:weighted-torus}
By Lemma~\ref{lem2-3}, we work with the admissible restriction
$\Pi|_{G_1}$.

Choose $a\in\mathcal O_{F_0}^\times$ such that
$\mathcal O_{F_0}=\Z_p[a]$ and its residue is a nonsquare. Such an $a$
can be chosen as the Teichm\"uller lift of a generator of the
multiplicative group of the residue field. For $G=D_F^\times$, write,
as in \cite[\S6.2]{bib9},
\[
 D_F=F\langle\vartheta,\sqrt a\rangle,\qquad
 \vartheta^2=\pi_F,\qquad(\sqrt a)^2=a,\qquad
 \vartheta\sqrt a=-\sqrt a\vartheta.
\]
For a sufficiently large integer $b\geq2$, take
\begin{equation}\label{eq3-1}
 \mathfrak h=p^b\bigl(\mathcal O_F\vartheta\oplus
 \mathcal O_F\sqrt a\,\vartheta\oplus\pi_F\mathcal O_F\sqrt a\bigr).
\end{equation}
For $G=\GL_2(F)$, put
\[
 e_+=\begin{pmatrix}0&1\\0&0\end{pmatrix},\qquad
 e_-=\begin{pmatrix}0&0\\1&0\end{pmatrix},\qquad
 h=\begin{pmatrix}1&0\\0&-1\end{pmatrix},
\]
and take
\begin{equation}\label{eq3-2}
 \mathfrak h=p^b\bigl(\mathcal O_Fe_+\oplus\pi_F\mathcal O_Fe_-
 \oplus\pi_F\mathcal O_Fh\bigr),\qquad\vartheta=e_++\pi_Fe_-.
\end{equation}
Set
\begin{equation}\label{eq3-3}
 H=\exp\mathfrak h,\qquad
 T=\exp(p^b\mathcal O_F\vartheta),\qquad T_n=T^{p^n}.
\end{equation}
The map $u\mapsto\exp(p^bu\vartheta)$ identifies
$(\mathcal O_F,+)$ with $T$ and $p^n\mathcal O_F$ with $T_n$.
Thus $T\simeq\Z_p^{d}$, $T_n\subseteq H_n$, and
$[T:T_n]=p^{dn}$.

Let $v_1,v_2,v_3$ be the ordered $\mathcal O_F$-basis of $\mathfrak h$ in
\eqref{eq3-1} or \eqref{eq3-2}, including the factor $p^b$.
As in \cite[\S6.5, (21)]{bib9}, use the following ordered basis of $H$:
\[
 q_{\ell,i,j}=\exp(a^i\pi_F^jv_\ell),
 \qquad 1\leq\ell\leq3,\quad0\leq i<f,\quad0\leq j<e.
\]
Denote the families by $(e_{i,j},f_{i,j},h_{i,j})$ for $\GL_2(F)$
and $(a_{i,j},b_{i,j},c_{i,j})$ for $D_F^\times$. Assign the weights
\begin{equation}\label{eq3-4}
\begin{gathered}
 w_*=\frac1{p-1}+\frac1{4e},\qquad
 w_j=w_*+\frac je,\qquad
 w'_j=w_*+\frac{j+1/2}e,\\
 \omega_H(q_{1,i,j})=\omega_H(q_{2,i,j})=w_j,\qquad
 \omega_H(q_{3,i,j})=w'_j.
\end{gathered}
\end{equation}
These satisfy $2w'_j=2w_j+1/e$ and lie in
$(1/(p-1),p/(p-1))$. For $b$ sufficiently large,
$J=\exp(p^{-1}\mathfrak h)$ is strictly saturated with the same
basis weights, and
\[
 H=J^p,\qquad\omega_H=\omega_J|_H-1.
\]
The results of Subsections~\ref{subsec:distribution-completions}
and~\ref{subsec:coadmissible-lattices} apply.

Write $D_{T,n}=D_{r_n}(T,E)$ with the restricted weights.

\subsection{Casimir operators and the associated graded module}
\label{subsec:casimir-residual}
Retain the lattice and filtrations of
Subsection~\ref{subsec:coadmissible-lattices}.
Let $\operatorname{Frob}$ be arithmetic Frobenius on
$F_0$ in the $\GL_2$ case and on $F_0(\sqrt a)$ otherwise. Put
$\tau_k=\operatorname{Frob}^k|_{F_0}$, and let $\beta_i$ be the
integral trace-dual basis to $a^i$, $0\leq i<f$.
Following \cite[Lem. 6.3]{bib9}, in the $\GL_2$ case set
\[
 \mathbf x_{k,j}
 =\sigma_{r_n}\left(\sum_{i=0}^{f-1}
       \operatorname{Frob}^k(\beta_i)(x_{i,j}-1)\right),
 \qquad x\in\{e,f,h\}.
\]
In the quaternion case define
\[
\begin{aligned}
 \mathbf w_{k,j}
 &=\frac12\sigma_{r_n}\left(
   \sum_{i=0}^{f-1}\left(
   \operatorname{Frob}^k(\beta_i)(a_{i,j}-1)
   +\operatorname{Frob}^k\left(\frac{\beta_i}{\sqrt a}\right)(b_{i,j}-1)
   \right)\right),\\
 \mathbf h_{\ell,j}
 &=\sigma_{r_n}\left(\sum_{i=0}^{f-1}
   \operatorname{Frob}^\ell\left(\frac{\beta_i}{\sqrt a}\right)(c_{i,j}-1)\right),
 \qquad 0\leq k<2f,\quad0\leq\ell<f.
\end{aligned}
\]
Here $0\leq j<e$, and in the $\GL_2$ case $0\leq k<f$.
We use $\mathbf x_{k,j}$ collectively for
$\mathbf e_{k,j},\mathbf f_{k,j},\mathbf h_{k,j}$ in the $\GL_2$ case
and $\mathbf w_{k,j},\mathbf w_{k+f,j},\mathbf h_{k,j}$ in the quaternion
case, with $0\leq k<f$.
The reductions $\overline{\mathbf x}_{k,j}$ modulo
$\eps=\sigma(\varpi)$ are independent of $n$ after rescaling the
grading \cite[Prop. 2.1 and Rem. 6.1]{bib9}, and
\[
 \bSH=\gr\overline\Lambda_H=
 \begin{cases}
 k_E[\overline{\mathbf e}_{k,j},\overline{\mathbf f}_{k,j},
     \overline{\mathbf h}_{k,j}:0\leq k<f,\ 0\leq j<e],
       &G=\GL_2(F),\\
 k_E[\overline{\mathbf w}_{k,j},\overline{\mathbf w}_{k+f,j},
     \overline{\mathbf h}_{k,j}:0\leq k<f,\ 0\leq j<e],
       &G=D_F^\times.
 \end{cases}
\]
The first two variables in each triple have weighted degree $w_j$,
and the third has degree $w'_j$.

For $0\leq k<f$ and $0\leq j<e$, put
\[
 u_{k,j}=
 \begin{cases}
 \overline{\mathbf e}_{k,j}+\overline{\mathbf f}_{k,j},
     &G=\GL_2(F),\\
 \overline{\mathbf w}_{k,j}+\overline{\mathbf w}_{k+f,j},
     &G=D_F^\times.
 \end{cases}
\]
Put $t_{i,j}=\exp(p^ba^i\pi_F^j\vartheta)$ and
$\overline{\mathbf t}_{i,j}=\sigma_{r_n}(t_{i,j}-1)\bmod\eps$,
for $0\leq i<f$ and $0\leq j<e$.
At $r_0=1/p$, the logarithm has the same symbol as $g-1$ for these
generators, since $w_j>1/(p-1)$
\cite[proof of Lem. 6.5]{bib9}.
Thus $\log t_{i,j}=\log e_{i,j}+\log f_{i,j}$ gives
\[
 \sigma_{r_0}(t_{i,j}-1)
 =\sigma_{r_0}(e_{i,j}-1)+\sigma_{r_0}(f_{i,j}-1)
\]
in the $\GL_2$ case.
In the quaternion case $t_{i,j}=a_{i,j}$, and
$\operatorname{Frob}^{f}(\sqrt a)=-\sqrt a$ makes the $b_{i,j}$
terms cancel in $\mathbf w_{k,j}+\mathbf w_{k+f,j}$.
After rescaling the degrees as in \eqref{eq2-4} and using the
trace-dual basis, we obtain
\[
\begin{aligned}
 u_{k,j}
 &=\sigma_{r_n}\left(\sum_{i=0}^{f-1}
   \operatorname{Frob}^k(\beta_i)(t_{i,j}-1)\right)\bmod\eps\\
 &=\sum_{i=0}^{f-1}\overline{\operatorname{Frob}^k(\beta_i)}
       \,\overline{\mathbf t}_{i,j},\\
 \overline{\mathbf t}_{i,j}
 &=\sum_{k=0}^{f-1}\overline{\tau_k(a^i)}\,u_{k,j}.
\end{aligned}
\]
The two inverse changes of basis above show that
\[
 \bST=\gr\overline\Lambda_T=k_E[u_{k,j}:0\leq k<f,\ 0\leq j<e]
 \subseteq\bSH.
\]
Let $P(t)$ be the Eisenstein polynomial of $\pi_F$ over $F_0$, and write
\[
 \frac{P(t)}{(t-\pi_F)P'(\pi_F)}
 =\sum_{j=0}^{e-1}\gamma_jt^j.
\]
Define $R_F\in\Z$ and, for $\rho:F\hookrightarrow E$, $\mu_{j,\rho}$ by
\[
 \frac{R_F}{e}=v_p(P'(\pi_F))+\frac1e-1,\qquad
 \mu_{j,\rho}=\rho(\pi_F^{R_F+j}\gamma_j).
\]
By \cite[Lem. 6.2]{bib9}, $v_p(\gamma_j)=-(R_F+j)/e$, so
$\mu_{j,\rho}\in\mathcal O_E^\times$; write $\overline\mu_{j,\rho}$
for its residue.

For each embedding $\rho:F\hookrightarrow E$, choose an
$\mathfrak{sl}_2(E)$-triple $(e_\rho,f_\rho,h_\rho)$ in the summand
$E\otimes_{F,\rho}\operatorname{Lie}_F(G_1)$ of the decomposition in
Subsection~\ref{subsec:conductors-radii}. In the $\GL_2$ case take the
images of the standard matrices $e_+,e_-,h$; in the quaternion case
use the normalization of \cite[\S6.3, (15)]{bib9}.
Here $\operatorname{Lie}_F(G_1)$ is the $F$-Lie algebra of $G_1$. Put
\[
 \Delta_\rho=\frac12h_\rho^2+e_\rho f_\rho+f_\rho e_\rho.
\]
The following polynomials occur in the calculation of the principal symbols of
the scaled Casimir operators \cite[proof of Thm. 6.2, (22)--(23)]{bib9}.
For $\rho:F\hookrightarrow E$ with $\rho|_{F_0}=\tau_k$, let $t$ be
a formal variable. In the $\GL_2$ case define
$Q_\rho(t)\in\bSH[t]$ by \cite[(22)]{bib9}:
\[
\begin{aligned}
 Q_\rho(t)
 &=\frac12\left(\sum_{j=0}^{e-1}\overline\mu_{e-1-j,\rho}
       \overline{\mathbf h}_{k,e-1-j}t^j\right)^2\\
 &\quad+2t\left(\sum_{j=0}^{e-1}\overline\mu_{e-1-j,\rho}
       \overline{\mathbf e}_{k,e-1-j}t^j\right)
       \left(\sum_{j=0}^{e-1}\overline\mu_{e-1-j,\rho}
       \overline{\mathbf f}_{k,e-1-j}t^j\right).
\end{aligned}
\]
In the quaternion case replace
$\overline{\mathbf e}_{k,j},\overline{\mathbf f}_{k,j}$ by
$\overline{\mathbf w}_{k,j},\overline{\mathbf w}_{k+f,j}$,
respectively, as in that formula.
For each $k$, choose an extension $\rho_k:F\hookrightarrow E$ of
$\tau_k$ and define the ideal $I$ as in
\cite[(23) and (24)]{bib9}:
\begin{equation}\label{eq3-6}
\begin{aligned}
 Q_\rho(t)&=\sum_{\ell=0}^{2e-1}\overline c_{\rho,\ell}t^\ell,\\
 I&=(\overline c_{\rho_k,\ell}:0\leq k<f,\ 0\leq\ell<2e)
       \subseteq\bSH.
\end{aligned}
\end{equation}
Each $\overline c_{\rho,\ell}$ is a quadratic polynomial in the
generators $\overline{\mathbf x}_{k,j}$ of $\overline S_H$, of weighted
degree $2w_*+(2e-1-\ell)/e$.
Since $\pi_F^{R_F+j}\gamma_j\in\mathcal O_F^\times$, its residue under
$\rho$ depends only on $\rho|_{F_0}$. Hence $I$ is independent of the
extensions $\rho_k$ \cite[Rem. 6.2]{bib9}.

In Proposition~\ref{prop3-3} below, we will combine the following
nilpotence result with Lemma~\ref{lem3-2} to prove that
$\gr\overline M$ is finite over $\bST$.

\begin{proposition}\label{prop3-1}
There is an integer $s\geq1$ such that $I^s\gr\overline M=0$
\cite[proof of Thm. 6.2, (24)]{bib9}.
\end{proposition}

\subsection{Finiteness and dimension comparison}
\label{subsec:torus-finiteness}
The following lemma refines the dimension estimate of
\cite[Lem. 6.6]{bib9} by using the torus subalgebra $\bST$.
\begin{lemma}\label{lem3-2}
$\bSH/I$ is a finite free $\bST$-module of rank $4^{d}$.
\end{lemma}
\begin{proof}
Give every variable of $\bSH$ ordinary degree one.
Put $A=\bSH/I$. For each $k$, write $H_k(t),E_k(t),F_k(t)$ for the
three sums in $Q_{\rho_k}(t)$, so that
$Q_{\rho_k}(t)=\frac12H_k(t)^2+2tE_k(t)F_k(t)$.
At a common geometric zero of the $\overline c_{\rho_k,\ell}$ and
$u_{k,j}$, we have $F_k(t)=-E_k(t)$, and \eqref{eq3-6} gives
\[
 H_k(t)^2=4tE_k(t)^2.
\]
If $E_k\ne0$, the two sides have orders of vanishing at $t=0$ of
opposite parity.
Thus $H_k=E_k=F_k=0$, and the unit coefficients
$\overline\mu_{j,\rho_k}$ force all coordinates to vanish.
Hence $A/(u_{k,j})$ is finite-dimensional.

The $2d$ quadrics $\overline c_{\rho_k,\ell}$ and the $d$
linear forms $u_{k,j}$ are therefore homogeneous parameters in the
Cohen--Macaulay ring $\bSH$. By homogeneity and
\cite[Lem. 10.104.2]{bib14} at the homogeneous maximal ideal, they
form a regular sequence. The Hilbert series are
\cite[\S\S10.56 and 10.58]{bib14}
\[
 H_A(q)=\frac{(1+q)^{2d}}{(1-q)^{d}},\qquad
 H_{A/(u_{k,j})}(q)=(1+q)^{2d}.
\]
Lifting a homogeneous basis of $A/(u_{k,j})$, with degrees $a_\ell$,
gives a surjection $\bigoplus_\ell\bST(-a_\ell)\twoheadrightarrow A$
by induction on degree. The Hilbert series of the source is
\[
\begin{aligned}
 H_{\bigoplus_\ell\bST(-a_\ell)}(q)
 &=H_{\bST}(q)\sum_\ell q^{a_\ell}\\
 &=\frac{H_{A/(u_{k,j})}(q)}{(1-q)^{d}}\\
 &=\frac{(1+q)^{2d}}{(1-q)^{d}}=H_A(q).
\end{aligned}
\]
Both sides have the same Hilbert series, so the surjection is an isomorphism. Thus $A$ is free over $\bST$ of rank $H_{A/(u_{k,j})}(1)=4^{d}$.
\end{proof}

\begin{proposition}\label{prop3-3}
Put $n_*=n_{\mathrm{res}}$ from \eqref{eq2-4}. Then $M$ is finite over
$\Lambda_T$ and, for every $n\geq n_*$, the natural map
\[
 D_{T,n}\otimes_{\Lambda_T}M\longrightarrow
 D_{H,n}\otimes_{\Lambda_H}M=M_n
\]
is an isomorphism. There is a positive integer $m=m(H,M)$ such that
every $M_n$, $n\geq n_*$, has at most $m$ generators over $D_{T,n}$.
\end{proposition}
\begin{proof}
The ordered basis $(t_{i,j})$ of $T$ extends to one of $H$,
replacing $v_1$ by
$v_1+v_2$ in the $\GL_2$ case and preserving the weights.
The ordered expansion formula
\cite[discussion following Thm. 4.10]{bib6} gives
\begin{equation}\label{eq3-5}
\begin{aligned}
 \operatorname{Fil}_{r_n}^s\Lambda_T
 &=\Lambda_T\cap\operatorname{Fil}_{r_n}^s\Lambda_H,\\
 \operatorname{Fil}_{r_n}^s\Lambda_T^\circ
 &=\Lambda_T^\circ\cap\operatorname{Fil}_{r_n}^s\Lambda_H^\circ
\end{aligned}
\end{equation}
for all $n\geq0$ and $s$, and hence the filtered embedding
$\overline\Lambda_T\hookrightarrow\overline\Lambda_H$.

By Proposition~\ref{prop3-1}, the $I$-adic filtration of $\gr\overline M$ has finitely many nonzero
successive quotients. Each
$I^a\gr\overline M/I^{a+1}\gr\overline M$ is finite over $\bSH/I$,
hence over $\bST$ by Lemma~\ref{lem3-2}. Thus $\gr\overline M$ is
finite over $\bST$.
Choose a nonempty finite family of homogeneous $\overline S_T$-generators
$\overline m_1,\ldots,\overline m_m$ of $\gr\overline M$.
Fix $n\geq n_*$, write $\operatorname{Fil}_n$ for the filtrations
at radius $r_n$, and lift these generators to homogeneous
$\widetilde m_{j,n}\in\gr_{r_n}M^\circ$ of degrees $d_{j,n}$, then to
$m_{j,n}\in M^\circ$.
By \eqref{eq2-4}, the $\widetilde m_{j,n}$ generate $\gr_{r_n}M^\circ$ over
both $\overline S_T[\eps]$ and $\overline S_H[\eps]$: the cokernels
vanish modulo $\eps$, have gradings bounded below, and $\deg\eps>0$.
Completeness of the filtrations with positive weights gives strict surjections
\[
 \phi_{K,n}:\bigoplus_{j=1}^m\Lambda_K^\circ(-d_{j,n})
 \twoheadrightarrow M^\circ,\qquad
 (a_j)_j\longmapsto\sum_ja_jm_{j,n},\quad K=T,H.
\]
By \eqref{eq3-5} and \cite[Lem. 2.2(ii)]{bib9}, the maps $\phi_{T,n}[1/p]$ and
$\phi_{H,n}[1/p]$ induce the same filtration on $M$:
\[
 \operatorname{Fil}_n^sM
 =\sum_j\operatorname{Fil}_n^{s-d_{j,n}}\Lambda_Tm_{j,n}
 =\sum_j\operatorname{Fil}_n^{s-d_{j,n}}\Lambda_Hm_{j,n}.
\]
Thus $M$ is finite over $\Lambda_T$. Let $\widehat M_n$ be the completion of $M$ for this filtration.
By \cite[Lem. 2.1]{bib9}, the following maps are isomorphisms:
\begin{equation}\label{eq3-8}
 D_{T,n}\otimes_{\Lambda_T}M\xrightarrow{\sim}\widehat M_n
 \xleftarrow{\sim}D_{H,n}\otimes_{\Lambda_H}M=M_n
\end{equation}
Both maps extend the identity on $M$.
Tensoring $\phi_{T,n}[1/p]$ gives $D_{T,n}^m\twoheadrightarrow M_n$.
Here $m$ is the number of the chosen homogeneous generators of
$\gr\overline M$ over $\overline S_T$, so the bound is uniform in $n$.
\end{proof}

\begin{corollary}\label{cor3-4}
For the analytic torus $T$ of \eqref{eq3-3}, the restriction $\Pi|_T$ is an admissible Banach representation, and the natural inclusions identify
\[
 (\Pi|_T)^{T\text{-}\an}=(\Pi|_H)^{H\text{-}\an}=\Pi^{G\text{-}\an}
\]
as locally convex $T$-representations.
\end{corollary}
\begin{proof}
Proposition~\ref{prop3-3} gives admissibility.
The isomorphisms \eqref{eq3-8} extend the identity on $M$ and commute
with transition maps, so \eqref{eq2-3} gives
\[
 D(T,E)\otimes_{\Lambda_T}M\simeq D(H,E)\otimes_{\Lambda_H}M
\]
as strong duals. Dualizing and applying Lemma~\ref{lem2-3} gives
the asserted identifications.
\end{proof}

\begin{corollary}\label{cor3-5}
Every coadmissible $D(H,E)$-subquotient $Q$ of $M_{\an}$ is coadmissible over $D(T,E)$. For every $n\geq n_*$, its radius modules are canonically identified by
\[
 D_{T,n}\otimes_{D(T,E)}Q\simeq D_{H,n}\otimes_{D(H,E)}Q=Q_n.
\]
In particular, $D_{T,n}\otimes_{D_{T,n+1}}Q_{n+1}\simeq Q_n$.
\end{corollary}
\begin{proof}
Corollary~\ref{cor3-4} shows that the $D(T,E)$- and $D(H,E)$-module structures on $M_{\an}$ define the same topology. Closed submodules and their quotients are
therefore coadmissible
over $D(T,E)$ \cite[Lem. 3.6]{bib6}.
For a coadmissible $D(H,E)$-submodule $L\subseteq M_{\an}$, exactness and
\eqref{eq3-8} embed both radius modules in $M_n$.
Both images are finite $D_{T,n}$-submodules of $M_n$, hence closed. The image
of $L$ is dense in each of them \cite[Prop. 2.1(ii)--(iii), Rem. 3.2, and
\S3, Thm. A]{bib6}. Thus
\[
 D_{T,n}\otimes_{D(T,E)}L=\overline{\operatorname{im}(L\longrightarrow M_n)}=L_n.
\]
Exactness gives the same identification for subquotients $Q$, and
coadmissibility gives the transition identities.
\end{proof}

For $p>2$, let $Q$ be a coadmissible subquotient of $M_{\an}$.
By Corollary~\ref{cor3-5}, it is coadmissible over $D(T,E)$, with
radius modules $Q_n$. Write $\delta_T(Q)$ for its canonical dimension
over $D(T,E)$, with $\delta_T(0)=-\infty$.
Since $D_{T,n}$ is regular of pure dimension $d$, canonical dimension
at each radius agrees with support dimension. Hence
\[
 \delta_T(Q)=d_{D(T,E)}(Q)
 =\sup_{n\geq n_*}\dim\operatorname{Supp}_{D_{T,n}}Q_n.
\]
By \cite[\S8, condition (DIM) and the discussion following it]{bib6},
this dimension function is exact. The support dimensions of the radius
modules over $D_{T,n}$ are nondecreasing and eventually equal to
$\delta_T(Q)$.
Moreover, $\delta_T(Q)=-\infty$ exactly when $Q=0$.

\begin{proposition}\label{prop3-6}
For every $n\geq n_*$, every $D_{H,n}$-subquotient $Q$ of $M_n$ satisfies
\[
 d_{D_{T,n}}(Q)=d_{D_{H,n}}(Q).
\]
Consequently, every coadmissible $D_H$-subquotient $Q$ of $M_{\an}$ satisfies
\[
 \delta_T(Q)=d_{D_H}(Q).
\]
\end{proposition}
\begin{proof}
Fix $n\geq n_*$. The first assertion is immediate for $Q=0$.
For $Q\ne0$, use the submodule and quotient filtrations induced from $M_n$.
The proof of Proposition~\ref{prop3-3} and~\eqref{eq2-4} show that
$\gr_{r_n}M_n$ is finite over $\gr_{r_n}D_{T,n}$.
Since this graded ring is Noetherian, $\gr_{r_n}Q$ is also finite over it.
Thus the same filtration on $Q$ is good over both $D_{T,n}$ and $D_{H,n}$.

Put $R=\gr_{r_n}D_{T,n}$, $S=\gr_{r_n}D_{H,n}$, and $N=\gr_{r_n}Q$.
The $R$-module $\operatorname{End}_R(N)$ is finite, and multiplication gives
an injection $S/\operatorname{Ann}_S(N)\hookrightarrow\operatorname{End}_R(N)$.
Hence $S/\operatorname{Ann}_S(N)$ is finite over $R/\operatorname{Ann}_R(N)$,
so the Krull dimensions of $N$ over $R$ and $S$ agree
\cite[Lem. 10.112.4]{bib14}.
Applying~\eqref{eq2-5} to $T$ and $H$ gives
\[
 d_{D_{T,n}}(Q)=\operatorname{Kdim}_R N-1
 =\operatorname{Kdim}_S N-1=d_{D_{H,n}}(Q).
\]
For a coadmissible subquotient, Corollary~\ref{cor3-5} identifies the two
radius modules. Taking the supremum over $n\geq n_*$ proves the second assertion.
\end{proof}

\section{Support over the torus and length bounds}
\label{sec4}
Let $p>2$, and retain the groups $H,T$ and the infinitesimal
character $\chi$ of $M_{\an}$ from Section~\ref{sec:analytic-torus}. By Proposition~\ref{prop3-3}, there is a fixed
positive integer $m$ such that $M_n$ is generated by at most $m$ elements over
$D_{T,n}$ for every $n\geq n_*$. Put $C_{\mathrm{fib}}=2^{d}m$.

Throughout this section, radius indices satisfy $n\geq n_*$ unless
explicitly stated otherwise. Every $D_{H,n}$-subquotient $Q$ of $M_n$ satisfies
\[
 \dim\operatorname{Supp}_{D_{T,n}}Q
 =d_{D_{T,n}}(Q)=d_{D_{H,n}}(Q)
\]
by Proposition~\ref{prop3-6}.

For a commutative affinoid $E$-algebra $A$, $\operatorname{Sp}A$ is its
rigid analytic space, and $\kappa(x)=A/\mathfrak m_x$ is the residue field
at a point $x$, with corresponding maximal ideal $\mathfrak m_x$.
For a finite $D_{T,n}$-module $Q$ and a point
$x\in\operatorname{Sp}D_{T,n}$, the fiber of $Q$ at $x$ is the
$\kappa(x)$-vector space
$Q\otimes_{D_{T,n}}\kappa(x)\simeq Q/\mathfrak m_xQ$. Put
\[
 \mu_x(Q)=\dim_{\kappa(x)}(Q/\mathfrak m_x Q).
\]
We use the same notation for fibers after coefficient extension.

\subsection{The torus distribution algebra and Lie algebra action}
\label{sec4-1}
Let $\sigma_1,\ldots,\sigma_d$ be the distinct
$\mathbf Q_p$-embeddings of $F$ into $E$. For $1\leq i\leq d$, let $\mathbf1_i$ be the
idempotent corresponding to $\sigma_i$ in
$E\otimes_{\mathbf Q_p}F\simeq\prod_{j=1}^{d}E$.
Choose $\eta_i\in E^\times$ with $\eta_i^2=\sigma_i(\pi_F)$ and define
the Lie algebra element $h_i\in E\otimes_{\mathbf Q_p}\operatorname{Lie}(T)$ by
\[
 h_i=\eta_i^{-1}\mathbf1_i(1\otimes\vartheta).
\]

For each $i$, identify the Lie algebra summand
$E\otimes_{F,\sigma_i}\operatorname{Lie}_F(G_1)$ of $\mathfrak g_E$
from Subsection~\ref{subsec:conductors-radii} with $\mathfrak{sl}_2(E)$
as follows. In the $\GL_2$ case apply $\sigma_i$ to matrix entries.
In the quaternion case, for $\sigma_i|_{F_0}=\tau_k$, use
\[
 \vartheta\longmapsto
       \begin{pmatrix}0&1\\\sigma_i(\pi_F)&0\end{pmatrix},\qquad
 \sqrt a\longmapsto
       \begin{pmatrix}\operatorname{Frob}^k(\sqrt a)&0\\0&-\operatorname{Frob}^k(\sqrt a)\end{pmatrix}.
\]
This sends the generators of \cite[(15)]{bib9} to
$(\sigma_i(\pi_F)^{-1}e_+,\sigma_i(\pi_F)e_-,h)$ and identifies
$\Delta_{\sigma_i}$ with the standard Casimir of $\mathfrak{sl}_2(E)$.

In this realization, the prescribed $h_i$ and our choices of $e_i,f_i$ are
\[
 h_i=\begin{pmatrix}0&\eta_i^{-1}\\\eta_i&0\end{pmatrix},\qquad
 e_i=\frac12\begin{pmatrix}1&-\eta_i^{-1}\\\eta_i&-1\end{pmatrix},\qquad
 f_i=\frac12\begin{pmatrix}1&\eta_i^{-1}\\-\eta_i&-1\end{pmatrix}.
\]
The triple $(e_i,f_i,h_i)$ is obtained by conjugating the triple
$(e_+,e_-,h)$ displayed in Subsection~\ref{subsec:weighted-torus} by
$\left(\begin{smallmatrix}1&1\\\eta_i&-\eta_i\end{smallmatrix}\right)$,
whose determinant is $-2\eta_i\ne0$. In particular,
\[
 [h_i,e_i]=2e_i,\qquad [h_i,f_i]=-2f_i,\qquad [e_i,f_i]=h_i.
\]
Elements with different indices lie in distinct
$\mathfrak{sl}_2(E)$-summands of $\mathfrak g_E$ and commute. Conjugation preserves the Casimir $\Delta_{\sigma_i}$ defined in
Subsection~\ref{subsec:casimir-residual}, so in the enveloping algebra
\[
 \frac12h_i^2+e_if_i+f_ie_i=\Delta_{\sigma_i}.
\]

Reindex the basis $(t_{i,j})$ of $T$ from
Subsection~\ref{subsec:casimir-residual} as $t_1,\ldots,t_d$, and enumerate
the corresponding elements $a^i\pi_F^j$, $0\leq i<f$, $0\leq j<e$,
as $\xi_1,\ldots,\xi_d$. Thus
\[
 t_\ell=\exp(p^b\xi_\ell\vartheta),\qquad
 w_\ell=w_*+v_p(\xi_\ell),
\]
with $w_*$ and the weights as in \eqref{eq3-4}.
Proposition~\ref{prop2-1} therefore applies to $T$.
For $n\geq0$, define
\[
 z_\ell=t_\ell-1,\qquad
 y_{\ell,n}=p^n\log t_\ell=p^n\log(1+z_\ell).
\]
For positive real numbers $R_1,\ldots,R_q$, write
$E\langle R_1^{-1}X_1,\ldots,R_q^{-1}X_q\rangle$ for the completion
of $E[X_1,\ldots,X_q]$ under the Gauss norm
\[
 \left\|\sum_\alpha a_\alpha X^\alpha\right\|
 =\sup_\alpha |a_\alpha|\prod_{i=1}^q R_i^{\alpha_i}.
\]
Here $\alpha\in\mathbf N^q$, $X^\alpha=\prod_iX_i^{\alpha_i}$, and
$|\alpha|=\sum_i\alpha_i$. Its elements are series for which
$|a_\alpha|\prod_iR_i^{\alpha_i}\to0$ as $|\alpha|\to\infty$.
 When all $R_i=1$, this is the usual Tate algebra
$E\langle X_1,\ldots,X_q\rangle$.

The elements $z_\ell$ give the distribution algebra presentation
\[
 D_{T,n}=E\langle r_n^{-w_1}z_1,\ldots,r_n^{-w_{d}}z_{d}\rangle.
\]
The completed enveloping algebra
\[
 \widehat U_{T,n}
 =D_{1/p}(T_n,E;\omega_H|_{T_n}-n)\subseteq D_{T,n}
\]
is the completion of $E[y_{1,n},\ldots,y_{d,n}]$ for the weighted PBW
norm, by Proposition~\ref{prop2-1}:
\[
 \left\|\sum_\nu a_\nu y_n^\nu\right\|
 =\max_\nu |a_\nu|p^{-\sum_\ell w_\ell\nu_\ell},
 \qquad |a_\nu|p^{-\sum_\ell w_\ell\nu_\ell}\longrightarrow0.
\]
Here $y_n^\nu=\prod_\ell y_{\ell,n}^{\nu_\ell}$, and the limit is taken as
$|\nu|\to\infty$. In the basis $h_1,\ldots,h_d$ of
$E\otimes_{\Qp}\operatorname{Lie}(T)$, we have
\begin{equation}\label{eq4-1}
 y_{\ell,n}=p^{n+b}\sum_{i=1}^{d}\sigma_i(\xi_\ell)\eta_i h_i.
\end{equation}
This linear change of basis is invertible, so
$E[y_{1,n},\ldots,y_{d,n}]=E[h_1,\ldots,h_d]$.
By Proposition~\ref{prop2-1},
$D_{T,n}\simeq\widehat U_{T,n}\rtimes_{T_n}T$.
Since $T$ is abelian, this algebra is generated over $\widehat U_{T,n}$
by the commuting group elements $t_1,\ldots,t_d$, with
$t_\ell^{p^n}=\exp(p^n\log t_\ell)=\exp(y_{\ell,n})$.
Writing $X_\ell$ for the variable representing the ordered generator
$t_\ell$ of $T$ gives
\begin{equation}\label{eq4-2}
 D_{T,n}\simeq
 \frac{\widehat U_{T,n}[X_1,\ldots,X_{d}]}
 {(X_\ell^{p^n}-\exp(y_{\ell,n}):1\leq\ell\leq d)},
 \qquad X_\ell\longmapsto t_\ell.
\end{equation}
The exponentials converge because $w_\ell>1/(p-1)$. Both sides are free
over $\widehat U_{T,n}$ with basis
$X_1^{a_1}\cdots X_{d}^{a_{d}}$, $0\leq a_\ell<p^n$, corresponding
to representatives of $T/T_n$. Each $X_\ell$ is a unit, so the derivatives
$p^nX_\ell^{p^n-1}$ are invertible. Thus this extension is finite
\'{e}tale \cite[Lem. 10.144.2]{bib14}.

Choose $\gamma_\ell^T\in E^\times$ with $|\gamma_\ell^T|=p^{-w_\ell}$.
Then
\begin{equation}\label{eq4-3}
 \widehat U_{T,n}
 =E\langle(\gamma_1^T)^{-1}y_{1,n},\ldots,
           (\gamma_d^T)^{-1}y_{d,n}\rangle.
\end{equation}

\begin{lemma}\label{lem4-1}
Let $n\geq0$, $I\subseteq\{1,\ldots,d\}$, and $\lambda\in E^I$.
If $\widehat U_{T,n}/(h_i-\lambda_i:i\in I)$ is nonzero, then
\begin{equation}\label{eq4-7}
 \widehat U_{T,n}/(h_i-\lambda_i:i\in I)
 \simeq E\langle V_1,\ldots,V_{d-|I|}\rangle.
\end{equation}
Here the $V_j$ are indeterminates.
\end{lemma}
\begin{proof}
Put $Y_\ell=(\gamma_\ell^T)^{-1}y_{\ell,n}$. By \eqref{eq4-1}, the affine
equations $h_i-\lambda_i$ have independent linear parts.
For $I\ne\varnothing$, choose a coefficient of maximal absolute value among the coefficients of
$Y_1,\ldots,Y_d$ in one equation. Divide by it and relabel the variables:
\[
 g=c+Y_1+\sum_{\ell=2}^d d_\ell Y_\ell,\qquad |d_\ell|\leq1.
\]
If $|c|>1$, then $g/c=1+u$ with $\|u\|<1$, so $g$ is a unit,
contrary to the assumption that the quotient is nonzero. Thus $|c|\leq1$, and bounded substitution gives
\[
 E\langle Y_1,\ldots,Y_d\rangle/(g)
 \xrightarrow[\;Y_1\mapsto-c-\sum_{\ell>1}d_\ell Y_\ell\;]{\ \sim\ }
 E\langle Y_2,\ldots,Y_d\rangle.
\]
The remaining $|I|-1$ equations still have independent linear parts,
and their quotient is nonzero. Induction on $|I|$, starting with
$I=\varnothing$, proves \eqref{eq4-7}.
\end{proof}

\subsection{Bounds over torus characters}
\label{subsec:torus-character-bounds}
For a field $L$, an $L$-vector space $Q$, and an $L$-linear
endomorphism $u:Q\to Q$, write $Q[u]=\{q\in Q:u(q)=0\}$.
Given an $L$-linear endomorphism $h$, view $Q$ as an $L[t]$-module
with $t$ acting as $h$. Its torsion submodule is
\[
 Q_{\mathrm{tor}}=\{q\in Q:P(h)q=0\text{ for some }0\ne P\in L[t]\}.
\]

\begin{lemma}\label{lem4-2}
Let $L$ be a field of characteristic zero, let $R$ be a commutative
Noetherian $L$-algebra, and let $Q$ be finite over $R$. Suppose $Q$ has
an $\mathfrak{sl}_2(L)$-action. Choose a basis $(e,f,h)$ of
$\mathfrak{sl}_2(L)$ satisfying
\[
 [h,e]=2e,\qquad [h,f]=-2f,\qquad [e,f]=h.
\]
Assume that $h$ acts by multiplication by an element of $R$, also
denoted by $h$. Then $Q_{\mathrm{tor}}$ is a locally finite
$\mathfrak{sl}_2(L)$-module,
and for every $\lambda\in L$,
\[
 Q[(h-\lambda)^r]=Q[h-\lambda]\qquad(r\geq1).
\]
If the Casimir $\Delta=\frac12h^2+ef+fe$ acts on $Q_{\mathrm{tor}}$ by a scalar $c$, then
$Q_{\mathrm{tor}}=0$ unless $c=k(k+2)/2$ for some integer $k\geq0$.
In that case the eigenvalues of $h$ on $Q_{\mathrm{tor}}$ belong to $\{k,k-2,\ldots,-k\}$.
\end{lemma}
\begin{proof}
We use the same letters $e,f,h$ for their actions on $Q$.
The $R$-submodule $Q_{\mathrm{tor}}$ has finitely many generators
$q_1,\ldots,q_t$. If $P_j(h)q_j=0$, then $P=\prod_jP_j$ satisfies
$P(h)Q_{\mathrm{tor}}=0$. The relations
\[
 (h-(\lambda+2))e=e(h-\lambda),\qquad
 (h-(\lambda-2))f=f(h-\lambda)
\]
make $Q_{\mathrm{tor}}$ stable under the Lie algebra action. After a
finite field extension splitting $P$, the Chinese remainder theorem gives
\[
 Q_{\mathrm{tor}}=\bigoplus_{\lambda\in\Lambda}Q_{(\lambda)},\qquad
 Q_{(\lambda)}=\{q:(h-\lambda)^rq=0\text{ for some }r\geq1\},
\]
where $\Lambda$ is the finite set of roots of $P$. Moreover,
\[
 eQ_{(\lambda)}\subseteq Q_{(\lambda+2)},\qquad
 fQ_{(\lambda)}\subseteq Q_{(\lambda-2)},\qquad
 Q_{(\mu)}=0\quad(\mu\notin\Lambda).
\]
Since $L$ has characteristic zero and $\Lambda$ is finite, some $N\geq1$
satisfies $e^NQ_{\mathrm{tor}}=f^NQ_{\mathrm{tor}}=0$.
Writing $\ell=\deg P$, PBW and $P(h)Q_{\mathrm{tor}}=0$ give
\[
 U(\mathfrak{sl}_2)q
 =\sum_{\substack{0\leq a,c<N\\0\leq b<\ell}}Le^ah^bf^cq
 \qquad(q\in Q_{\mathrm{tor}}).
\]
Each such submodule is finite-dimensional. By complete reducibility, $h$ acts semisimply on it. Since every element killed by
$(h-\lambda)^r$ belongs to $Q_{\mathrm{tor}}$, this gives
$Q[(h-\lambda)^r]=Q[h-\lambda]$.
This equality also holds over $L$, since scalar extension is faithful. On a simple
module of highest weight $k\geq0$, the Casimir is $k(k+2)/2$ and the
weights are $k,k-2,\ldots,-k$. These Casimir values are distinct, proving
the last assertion.
\end{proof}

Since distinct $\mathfrak{sl}_2(E)$-summands commute, imposing
$h_j=\lambda_j$ for $j\in J$ preserves the actions indexed by $i\notin J$.
Lemma~\ref{lem4-2} applies to the resulting quotient modules.

\begin{theorem}\label{thm4-3}
Let $n\geq n_*$ and let $x\in\operatorname{Sp}D_{T,n}$.
Every $D_{H,n}$-subquotient $Q$ of $M_n$ satisfies
\[
 \mu_x(Q)\leq C_{\mathrm{fib}}.
\]
\end{theorem}
\begin{proof}
A finite coefficient extension preserves fiber dimensions and the bound
$m$ on the number of generators over $D_{T,n}$, and is faithfully flat
\cite[Lem. 1.1.5]{bib15}.
For this proof, we may therefore assume that $x$ is $E$-rational.
For a subset $J\subseteq\{1,\ldots,d\}$, put
\[
\begin{gathered}
 R^{(J)}=D_{T,n}/(h_j-h_j(x):j\in J),\qquad r=d-|J|,\\
 M^{(J)}=M_n\otimes_{D_{T,n}}R^{(J)}.
\end{gathered}
\]
Then $M^{(J)}$ is generated by at most $m$ elements and carries the
commuting $\mathfrak{sl}_2(E)$-actions indexed by $i\notin J$, with
$h_i$ acting by multiplication.
In this induction, invariant subquotients are those stable under these
$\mathfrak{sl}_2(E)$-actions. We prove by induction on $r$ that
\[
 \mu_x(Q)\leq2^rm
 \quad\text{for every invariant subquotient }Q\text{ of }M^{(J)}.
\]
When $r=0$, Lemma~\ref{lem4-1} and the finite \'{e}tale presentation
\eqref{eq4-2} show that $R^{(J)}$ is finite \'{e}tale over $E$. Hence it
is a finite product of fields, with
$\kappa(x)=E$ as the factor corresponding to $x$. Projection onto this
factor is exact, so $\mu_x(Q)\leq\mu_x(M^{(J)})\leq m$.

Suppose $r>0$, and let $Q$ be an invariant subquotient of $M^{(J)}$.
Choose an invariant submodule $A_0\subseteq M^{(J)}$ with an equivariant
surjection $A_0\twoheadrightarrow Q$, and put $B_0=M^{(J)}/A_0$.
Choose $i\notin J$ and let
$t\in R^{(J)}$ be the image of the element $h_i-h_i(x)$.
Applying the snake lemma to multiplication by
$t$ on $0\to A_0\to M^{(J)}\to B_0\to0$ gives the exact sequence
\[
 M^{(J)}[t]\longrightarrow B_0[t]\longrightarrow A_0/tA_0
 \longrightarrow M^{(J)}/tM^{(J)}\longrightarrow B_0/tB_0
 \longrightarrow0.
\]
Lemma~\ref{lem4-2} gives $B_0[t^2]=B_0[t]$.
Thus $B_0[t]\cap tB_0=0$, so the natural map $B_0[t]\to B_0/tB_0$
is injective. Let $C$ be the kernel of
$A_0/tA_0\to M^{(J)}/tM^{(J)}$. Exactness gives
\[
 C\simeq B_0[t]\big/\operatorname{im}\bigl(M^{(J)}[t]\to B_0[t]\bigr).
\]
The injection $B_0[t]\hookrightarrow B_0/tB_0$ therefore makes $C$ a
subquotient of $B_0/tB_0$, and hence of $M^{(J)}/tM^{(J)}$.
Let $A_1$ be the image of $A_0/tA_0\to M^{(J)}/tM^{(J)}$.
After imposing $h_i=h_i(x)$, we work over $R^{(J\cup\{i\})}=R^{(J)}/tR^{(J)}$
with the module $M^{(J\cup\{i\})}=M^{(J)}/tM^{(J)}$.
The modules $C$ and $A_1$ are invariant under the $r-1$
$\mathfrak{sl}_2(E)$-actions with indices outside $J\cup\{i\}$. The induction hypothesis therefore applies
to both of them. Tensoring $0\to C\to A_0/tA_0\to A_1\to0$ with the
residue field $E$ gives
\[
 \mu_x(A_0)=\mu_x(A_0/tA_0)
 \leq\mu_x(C)+\mu_x(A_1)
 \leq2^{r-1}m+2^{r-1}m=2^rm.
\]
Right exactness of tensor product gives $\mu_x(Q)\leq\mu_x(A_0)$,
completing the induction. Taking $J=\varnothing$ gives the assertion,
since $M^{(\varnothing)}=M_n$.
\end{proof}

For subquotients of canonical dimension zero, the bound for fibers over
torus characters gives a dimension estimate in terms of the conductor.
\begin{corollary}\label{cor4-4}
Let $n\geq n_*$ and let $Q$ be a $D_{H,n}$-subquotient of $M_n$
of canonical dimension zero. Then the infinitesimal character $\chi$
is algebraic, with algebraic type $V$ as in Subsection~\ref{subsec:conductors-radii}.
There is a finite-dimensional smooth $H$-representation $W$ such that
$Q=V^\vee\otimes_EW^\vee$ and
\[
 \dim_EW\leq C_{\mathrm{fib}}p^{dc_H(W)}.
\]
\end{corollary}
\begin{proof}
By Lemma~\ref{lem2-2}, $Q$ is finite-dimensional over $E$.
Subsection~\ref{subsec:conductors-radii} shows that the existence of $Q\ne0$ forces $\chi$ to be
algebraic and gives the displayed decomposition. Put $\nu=c_H(W)$.
Finite coefficient extension preserves the dimensions and the bound $m$
on the number of generators over $D_{T,n}$, so, for this proof, we may
assume $\mu_{p^\nu}\subset E$.
The torus acts semisimply on $Q$,
through algebraic weights on $V^\vee$ and a finite abelian quotient
on $W^\vee$.
For each character $\theta$ of $T$ occurring in $Q$, denote the
corresponding eigenspace by $Q[\theta]$. Density of $E[T]$ in $D_{T,n}$
shows that $D_{T,n}$ acts by scalars on this eigenspace, giving a character
$D_{T,n}\to E$ and an $E$-rational point of $\operatorname{Sp}D_{T,n}$.
Semisimplicity gives
\[
 Q[\theta]\simeq Q\otimes_{D_{T,n},\theta}E.
\]
Thus Theorem~\ref{thm4-3} gives
$\dim_E Q[\theta]\leq C_{\mathrm{fib}}$.
There are at most $d_V=\dim_EV$ algebraic weights and
$[T:T_\nu]=p^{d\nu}$ characters of $T/T_\nu$, so
\[
 \#\{\theta:Q[\theta]\ne0\}\leq d_Vp^{d\nu}.
\]
Consequently,
\[
 d_V\dim_EW=\dim_E Q
 =\sum_\theta\dim_E Q[\theta]
 \leq C_{\mathrm{fib}}d_Vp^{d\nu}.
\]
\end{proof}

\subsection{Support components}
\label{sec4-2}
For $n\geq0$, $I\subseteq\{1,\ldots,d\}$ and
$\lambda=(\lambda_i)_{i\in I}\in E^I$, define the closed analytic subspace
\begin{equation}\label{eq4-6}
 Z_{I,\lambda,n}=\operatorname{Sp}\bigl(D_{T,n}/(h_i-\lambda_i:i\in I)\bigr).
\end{equation}
For $\lambda$ with entries in a coefficient extension, apply the same
definition over that field. The coordinate algebra of $Z_{I,\lambda,n}$ is
\[
 \mathcal O(Z_{I,\lambda,n})
 =D_{T,n}/(h_i-\lambda_i:i\in I)
 =D_{T,n}\otimes_{\widehat U_{T,n}}
        \widehat U_{T,n}/(h_i-\lambda_i:i\in I).
\]

Lemma~\ref{lem4-1} and the finite \'{e}tale presentation \eqref{eq4-2}
show that every nonempty $Z_{I,\lambda,n}$ is smooth of dimension $d-|I|$.
The same arguments apply after any complete coefficient extension.

Define the $E$-algebra automorphism $\mathcal T_i$ of $E[h_1,\ldots,h_d]$ by
\[
 \mathcal T_i(h_j)=h_j+2\delta_{ij}.
\]
By \eqref{eq4-1},
\[
 \mathcal T_i(y_{\ell,n})
 =y_{\ell,n}+2p^{n+b}\sigma_i(\xi_\ell)\eta_i.
\]
Since $v_p(\eta_i)=1/(2e)$ and $w_\ell=w_*+v_p(\xi_\ell)$, the choice of $b$ gives
\[
 v_p\bigl(2p^{n+b}\sigma_i(\xi_\ell)\eta_i\bigr)
 =n+b+v_p(2)+v_p(\xi_\ell)+\frac1{2e}\geq w_\ell.
\]
Thus $\mathcal T_i$ and its inverse are bounded for the PBW norm.
They therefore extend to mutually inverse continuous automorphisms of
$\widehat U_{T,n}$.
The bracket relations give
\[
 e_i h_j=(h_j-2\delta_{ij})e_i,\qquad
 f_i h_j=(h_j+2\delta_{ij})f_i.
\]
Since polynomials are dense in $\widehat U_{T,n}$, continuity gives the following identities for $r\in\widehat U_{T,n}$ and $q$ in
any $D_{T,n}$-subquotient of $M_n$ stable under the actions of $e_i$
and $f_i$:
\begin{equation}\label{eq4-4}
 e_i(rq)=\mathcal T_i^{-1}(r)e_iq,\qquad
 f_i(rq)=\mathcal T_i(r)f_iq.
\end{equation}
Let $c_i=\chi(\Delta_{\sigma_i})$ and define the polynomial $q_i\in E[t]$ by
\[
 q_i(t)=\frac{c_i}{2}-\frac{t^2}{4}-\frac{t}{2}.
\]
Since $\Delta_{\sigma_i}=\frac12h_i^2+2f_ie_i+h_i$, the following identities hold on every such subquotient:
\begin{equation}\label{eq4-5}
 f_ie_i=q_i(h_i),\qquad e_if_i=q_i(h_i-2).
\end{equation}
Put
\[
 I_{\mathrm{alg}}=\{i\in\{1,\ldots,d\}:
   c_i=k(k+2)/2\text{ for some }k\in\Z_{\geq0}\}.
\]
For $i\in I_{\mathrm{alg}}$, let $k_i$ be the unique such integer, and put
\[
 \Omega_i=\{k_i,k_i-2,\ldots,-k_i\}\quad(i\in I_{\mathrm{alg}}),
 \qquad \Omega_i=\varnothing\quad(i\notin I_{\mathrm{alg}}).
\]

For an irreducible closed analytic subspace
$C\subseteq\operatorname{Sp}D_{T,n}$, put
\[
 I(C)=\{i\in\{1,\ldots,d\}:h_i|_C\in E\},\qquad
 h_i|_C=\lambda_i\quad(i\in I(C)).
\]

\begin{proposition}\label{prop4-5}
Let $Q$ be a finite $D_{T,n}$-module with a $\mathfrak g_E$-action.
Suppose each $h_i$ acts by multiplication, the actions of $e_i$ and $f_i$
satisfy \eqref{eq4-4} for $r\in\widehat U_{T,n}$, and each
$\Delta_{\sigma_i}$ acts by $c_i$.
Let $C$ be an irreducible component of $\operatorname{Supp}_{D_{T,n}}(Q)$.
Then $I(C)\subseteq I_{\mathrm{alg}}$, $\lambda_i\in\Omega_i$ for
$i\in I(C)$, and $C$ is an irreducible component of
$Z_{I(C),\lambda,n}$, of dimension $d-|I(C)|$.
\end{proposition}
\begin{proof}
Put $\widehat U=\widehat U_{T,n}$.
Let $\mathfrak p_C\subset D_{T,n}$ be the prime defining $C$. It is minimal
in $\operatorname{Supp}_{D_{T,n}}(Q)$, hence associated to $Q$.
Choose $0\ne q\in Q$ with $\operatorname{Ann}_{D_{T,n}}(q)=\mathfrak p_C$.
For $i\in I(C)$ we have $(h_i-\lambda_i)q=0$, so
Lemma~\ref{lem4-2} gives $i\in I_{\mathrm{alg}}$ and
$\lambda_i\in\Omega_i$. Thus $C\subseteq Z_{I(C),\lambda,n}$ and
$\dim C\leq d-|I(C)|$.

Set
\[
 \mathfrak q=\mathfrak p_C\cap\widehat U,\qquad
 \mathfrak a=(h_i-\lambda_i:i\in I(C))\widehat U.
\]

Since $D_{T,n}$ is finite over $\widehat U$, the module $Q$ is finite over
$\widehat U$. By the definition of $\mathfrak q$,
$\operatorname{Ann}_{\widehat U}(q)=\mathfrak q$.
We prove that $\mathfrak a=\mathfrak q$ in three steps.

\emph{Step 1.} We prove that, for every $i\notin I(C)$, some
$N_i\geq1$ satisfies $\mathcal T_i^{N_i}(\mathfrak q)=\mathfrak q$.
Fix $i\notin I(C)$ and set
$\mathcal S_i=\{P(h_i):0\ne P\in E[t]\}$.
If $P(h_i)\in\mathfrak p_C$ for a nonzero $P\in E[t]$, then $P(h_i)q=0$.
Lemma~\ref{lem4-2} gives $i\in I_{\mathrm{alg}}$ and
\[
 \prod_{\lambda\in\Omega_i}(h_i-\lambda)q=0.
\]
Since $\operatorname{Ann}_{D_{T,n}}(q)=\mathfrak p_C$ is prime, some
$h_i-\lambda$ belongs to $\mathfrak p_C$, contradicting $i\notin I(C)$.
Thus $\mathcal S_i\cap\mathfrak p_C=\varnothing$, and hence
$\mathcal S_i\cap\mathfrak q=\varnothing$.
The identity $\mathcal T_i(h_i)=h_i+2$ shows that $\mathcal T_i$ preserves
$\mathcal S_i$. Put $\widehat U_i=\mathcal S_i^{-1}\widehat U$ and
$Q_i=\mathcal S_i^{-1}Q$. On $Q_i$, $e_i$ is $\mathcal T_i^{-1}$-semilinear over
$\widehat U_i$ by \eqref{eq4-4}.
Both $q_i(h_i)$ and $q_i(h_i-2)$ are units in $\widehat U_i$.
The identities in \eqref{eq4-5} show that $q_i(h_i)^{-1}f_i$ is the inverse of $e_i$:
\[
 \begin{aligned}
 (q_i(h_i)^{-1}f_i)e_i&=q_i(h_i)^{-1}q_i(h_i)=1,\\
 e_i(q_i(h_i)^{-1}f_i)&=q_i(h_i-2)^{-1}e_if_i=1.
 \end{aligned}
\]
In particular,
\[
 a\,e_i(v)=e_i(\mathcal T_i(a)v)\qquad(a\in\widehat U_i,\ v\in Q_i).
\]
Since $e_i$ is bijective,
\[
 \operatorname{Ann}_{\widehat U_i}(e_i(v))
 =\mathcal T_i^{-1}\bigl(\operatorname{Ann}_{\widehat U_i}(v)\bigr).
\]
Thus $\mathcal T_i$ permutes $\operatorname{Ass}_{\widehat U_i}(Q_i)$,
which is finite by \cite[Lem. 10.63.5]{bib14}.
The nonzero element $q/1\in Q_i$ has annihilator
$\mathfrak q\widehat U_i$, so this prime belongs to that set.
Hence some $N_i\geq1$ satisfies
\[
 \mathcal T_i^{N_i}(\mathfrak q\widehat U_i)=\mathfrak q\widehat U_i.
\]
Since $\mathcal T_i(\mathcal S_i)=\mathcal S_i$ and primes disjoint from
$\mathcal S_i$ are recovered by contraction to $\widehat U$, we obtain
\[
 \begin{aligned}
 \mathcal T_i^{N_i}(\mathfrak q)
 &=\mathcal T_i^{N_i}(\mathfrak q\widehat U_i)\cap\widehat U\\
 &=(\mathfrak q\widehat U_i)\cap\widehat U=\mathfrak q.
 \end{aligned}
\]

\emph{Step 2.} Write $y_\ell=y_{\ell,n}$ and use the weighted PBW norm.
Define the continuous derivation
\[
 \partial_i=2\frac{\partial}{\partial h_i}
 =\sum_{\ell=1}^{d}2p^{n+b}\sigma_i(\xi_\ell)\eta_i
       \frac{\partial}{\partial y_\ell}\qquad\text{on }\widehat U.
\]
We prove that $\partial_i(\mathfrak q)\subseteq\mathfrak q$ for
$i\notin I(C)$. The displayed expression for $\partial_i$ follows from \eqref{eq4-1}, and the derivatives $\partial/\partial y_\ell$ act termwise. In particular,
$\partial_i(h_j)=2\delta_{ij}$.
Put $K_i=\max_\ell p^{w_\ell}|\partial_i(y_\ell)|$.
For $m_a=N_ip^a$ with $a\geq0$, the translation formula is
\[
 \mathcal T_i^{m_a}(y_\ell)=y_\ell+m_a\partial_i(y_\ell).
\]
If $|m_a|K_i\leq1$, the binomial theorem gives
\[
 \left\|\frac{\mathcal T_i^{m_a}(f)-f}{m_a}-\partial_i(f)\right\|
 \leq |m_a|K_i^2\|f\|\xrightarrow[a\to\infty]{}0
 \qquad(f\in\widehat U).
\]
For a monomial $y^\nu$, a term containing $k\geq2$ factors
$m_a\partial_i(y_\ell)$ has, after division by $m_a$, norm at most
$|m_a|K_i^2(|m_a|K_i)^{k-2}\|y^\nu\|$.
The estimate extends from polynomials to $\widehat U$ by density.
For $f\in\mathfrak q$, every difference quotient belongs to $\mathfrak q$.
This ideal is closed by \eqref{eq4-3} and \cite[Cor. 5.2.7/2]{bib16},
so $\partial_i(\mathfrak q)\subseteq\mathfrak q$ for every $i\notin I(C)$.

\emph{Step 3.} We prove that the image of $\mathfrak q$ in
$\widehat U/\mathfrak a$ is zero. Put $r=d-|I(C)|$.
Since $\mathfrak a\subseteq\mathfrak q$ and
$\mathfrak q$ is proper, Lemma~\ref{lem4-1} gives
\[
 B=\widehat U/\mathfrak a\simeq E\langle V_1,\ldots,V_r\rangle.
\]
The proper ideal $\overline{\mathfrak q}=\mathfrak q/\mathfrak a$ is zero
if $r=0$, since then $B=E$. Assume $r>0$.
For $i\notin I(C)$, the derivation $\partial_i$ annihilates the
generators of $\mathfrak a$ and therefore descends to $B$.
By the successive substitutions in the proof of Lemma~\ref{lem4-1}, the images
of the $h_i$, $i\notin I(C)$, are affine linear functions of $V_1,\ldots,V_r$
with invertible coefficient matrix. The induced derivations $\partial_i$ therefore span the coordinate
derivations $\partial/\partial V_j$ over $E$. By Step 2,
$\overline{\mathfrak q}$ is stable under all these coordinate derivations.
Suppose $0\ne f=\sum_\mu a_\mu V^\mu\in\overline{\mathfrak q}$.
Among the indices for which $|a_\nu|$ is maximal, choose $\nu$ with
maximal total degree; there are only finitely many such indices because
$|a_\mu|\to0$. Then
\[
 \frac{1}{\nu!}
 \frac{\partial^{|\nu|}f}
      {\partial V_1^{\nu_1}\cdots\partial V_r^{\nu_r}}
 =a_\nu+\sum_{\beta\ne0}
       \binom{\nu+\beta}{\nu}a_{\nu+\beta}V^\beta
 \in\overline{\mathfrak q},
 \qquad \nu!=\prod_{j=1}^r\nu_j!.
\]
The binomial coefficients have absolute value at most one. By the
choice of $\nu$ and the decay of the coefficients, the nonconstant part has Gauss
norm strictly less than $|a_\nu|$. Thus this element is a unit,
contradicting the properness of $\overline{\mathfrak q}$.
Consequently $\overline{\mathfrak q}=0$ in all cases, so
$\mathfrak q=\mathfrak a$.

The inclusion $\widehat U/\mathfrak q\hookrightarrow D_{T,n}/\mathfrak p_C$
is finite, so \cite[Lem. 10.112.4]{bib14} gives
\[
 \dim C=\dim(\widehat U/\mathfrak q)=r
       =\dim Z_{I(C),\lambda,n}.
\]
By \cite[Cor. 2.2.7]{bib15}, $C$ is an irreducible component of
$Z_{I(C),\lambda,n}$.
\end{proof}

Consequently, every nonzero $D_{H,n}$-subquotient $Q$ of $M_n$ satisfies
\[
 \dim\operatorname{Supp}_{D_{T,n}}Q\geq d-|I_{\mathrm{alg}}|.
\]

\subsection{Geometric irreducibility of \texorpdfstring{$Z_{I,\lambda,n}$}{Z(I,lambda,n)}}
\label{sec4-3}
We use the following elementary Kummer criterion.

\begin{lemma}\label{lem4-7}
Let $A$ be a domain of characteristic zero, let $K=\operatorname{Frac}(A)$,
and let $N=p^n$ with $n\geq1$.
Suppose that $K$ contains all $N$th roots of unity and that
$u_1,\ldots,u_d\in A^\times$ satisfy
\[
 \prod_{\ell=1}^du_\ell^{a_\ell}\notin K^{\times p}
 \quad\text{for every }(a_1,\ldots,a_d)
 \in\{0,\ldots,p-1\}^d\setminus\{(0,\ldots,0)\},
\]
where $K^{\times p}=\{f^p:f\in K^\times\}$. Then
\[
 B=A[U_1,\ldots,U_d]/(U_\ell^N-u_\ell:1\leq\ell\leq d)
\]
is a domain, free of rank $N^d$ over $A$.
\end{lemma}
\begin{proof}
The monic relations give the $A$-basis $U_1^{a_1}\cdots U_d^{a_d}$,
$0\leq a_\ell<N$.
Choose $v_\ell$ in an algebraic closure of $K$ with $v_\ell^N=u_\ell$,
and put $K'=K(v_1,\ldots,v_d)$.
Since $K$ contains all $N$th roots of unity, $K'/K$ is Galois and
\[
 \operatorname{Gal}(K'/K)\hookrightarrow(\mu_N)^d,\qquad
 g\longmapsto\bigl(g(v_\ell)/v_\ell\bigr)_\ell,
\]
where $\mu_N\subset K^\times$ is the group of $N$th roots of unity.
If this map is not surjective, its cokernel is a nonzero finite abelian
$p$-group and therefore admits a nontrivial character to $\mu_p$.
Its pullback to $(\mu_N)^d$ has the form
\[
 (\zeta_1,\ldots,\zeta_d)\longmapsto
 \prod_{\ell=1}^d\zeta_\ell^{Na_\ell/p}
\]
for a nonzero tuple $(a_\ell)\in\{0,\ldots,p-1\}^d$.
Thus $f=\prod_\ell v_\ell^{Na_\ell/p}$ is fixed by
$\operatorname{Gal}(K'/K)$ and belongs to $K$, but
$f^p=\prod_\ell u_\ell^{a_\ell}$, contrary to the hypothesis.
Consequently $[K':K]=N^d$.
The surjection $B\otimes_AK\to K'$ defined by $U_\ell\mapsto v_\ell$
is an isomorphism, since both sides have $K$-dimension $N^d$.
Freeness over $A$ makes $B\to B\otimes_AK$ injective, so $B$ is a domain.
\end{proof}

\begin{proposition}\label{prop4-8}
Let $I\subseteq\{1,\ldots,d\}$ and $\lambda\in E^I$.
For every $n\geq0$, each nonempty $Z_{I,\lambda,n}$ is smooth of
dimension $d-|I|$ and is geometrically irreducible if $|I|<d$.
\end{proposition}
\begin{proof}
The finite \'{e}tale presentation reduces irreducibility to the Kummer
criterion. To rule out a $p$th root in the relevant fraction field, we
restrict to a one-dimensional disc and show that the coefficients of
such a root cannot tend to zero.
Fix $n\geq0$ with $Z_{I,\lambda,n}\ne\varnothing$, and put
$s=d-|I|$ and $N=p^n$.
Let $L/E$ be any finite extension containing $\mu_N$.
By \cite[Thm. 3.4.2]{bib15}, proving irreducibility after every such
extension suffices for geometric irreducibility.
Use the torus basis $t_\ell=\exp(p^b\xi_\ell\vartheta)$ and the weights
$w_\ell=w_*+v_p(\xi_\ell)$ fixed in Subsection~\ref{sec4-1}. Write
\[
 \begin{gathered}
 y_\ell=y_{\ell,n},\qquad
 \widehat U=L\otimes_E\widehat U_{T,n},\\
 \widehat U_{I,\lambda}=\widehat U/(h_i-\lambda_i:i\in I),\qquad
 B=L\otimes_E\mathcal O(Z_{I,\lambda,n}).
 \end{gathered}
\]
Since $B\ne0$, the algebra $\widehat U_{I,\lambda}$ is nonzero.
By Lemma~\ref{lem4-1} over $L$, it is a Tate algebra in $s$ variables, hence a
normal domain \cite[Thm. 5.2.6/2]{bib16}.
Quotienting \eqref{eq4-2} gives
\begin{equation}\label{eq4-8}
 B\simeq\frac{\widehat U_{I,\lambda}[X_1,\ldots,X_d]}
 {(X_\ell^N-\exp(y_\ell):1\leq\ell\leq d)}.
\end{equation}
Thus $B$ is finite \'{e}tale and free of rank $N^d$ over
$\widehat U_{I,\lambda}$. The smoothness and dimension assertions were
established after \eqref{eq4-6}.
For irreducibility, assume $s>0$. When $n=0$, the presentation gives
$B=\widehat U_{I,\lambda}$. It remains to consider $n\geq1$.

Put $K=\operatorname{Frac}(\widehat U_{I,\lambda})$.
In the quotient norm on $\widehat U_{I,\lambda}$ induced by the
weighted PBW norm on $\widehat U$, each $y_\ell$ has norm at most
$p^{-w_\ell}<p^{-1/(p-1)}$, so
$\prod_\ell\exp(y_\ell)^{a_\ell}=\exp(\sum_\ell a_\ell y_\ell)$.
By Lemma~\ref{lem4-7}, it is enough to prove
\begin{equation}\label{eq4-9}
 \exp\!\left(\sum_{\ell=1}^da_\ell y_\ell\right)\notin K^{\times p}
 \quad\text{for every }(a_1,\ldots,a_d)
 \in\{0,\ldots,p-1\}^d\setminus\{(0,\ldots,0)\}.
\end{equation}
Suppose, contrary to \eqref{eq4-9}, that
\[
 f^p=\exp\!\left(\sum_\ell a_\ell y_\ell\right)
 \qquad\text{for some }f\in K^\times.
\]
The displayed monic equation and normality give $f\in\widehat U_{I,\lambda}$.
Moreover, $f$ is a unit because $f^p$ is a unit.

Choose an $L$-rational point $x$ of $\operatorname{Sp}(\widehat U_{I,\lambda})$,
which exists by Lemma~\ref{lem4-1} over $L$.
Choose $i\notin I$ and $\gamma\in E^\times$ with $v_p(\gamma)=w_*$.
This valuation is chosen so that the exponential on the disc below
converges while its formal $p$th root has coefficients that do not
tend to zero.
The assignments
\begin{equation}\label{eq4-10}
 \varphi:\widehat U_{I,\lambda}\longrightarrow L\langle X\rangle,
 \qquad y_\ell\longmapsto y_\ell(x)+\gamma\sigma_i(\xi_\ell)X
\end{equation}
define a continuous $L$-algebra homomorphism.
Indeed, \eqref{eq4-1} gives
\[
 h_j\longmapsto h_j(x)+\delta_{ij}\frac{\gamma X}{p^{n+b}\eta_i},
\]
so the equations $h_j=\lambda_j$ for $j\in I$ are preserved.
The bounds
\[
 |y_\ell(x)|\leq p^{-w_\ell},\qquad
 |\gamma\sigma_i(\xi_\ell)|=p^{-w_\ell}
\]
make substitution bounded for the weighted PBW norm.
Put $u=\sum_\ell a_\ell \xi_\ell$.
Since $(\xi_\ell)$ is a $\mathbf Z_p$-basis of $\mathcal O_F$ and the
tuple $(a_\ell)$ is nonzero modulo $p$,
\[
 u\in\mathcal O_F\setminus p\mathcal O_F,\qquad
 0\leq v_p(u)\leq\frac{e-1}{e}.
\]
Since $f(x)\ne0$, the element $g=\varphi(f)/f(x)\in L\langle X\rangle$
satisfies
\[
 g(X)^p=
 \frac{\exp\bigl(\sum_\ell a_\ell y_\ell(x)+\gamma\sigma_i(u)X\bigr)}
 {\exp\bigl(\sum_\ell a_\ell y_\ell(x)\bigr)}
 =\exp\bigl(\gamma\sigma_i(u)X\bigr).
\]
Put $c=\gamma\sigma_i(u)$ and $\alpha=v_p(c)$.
The choice of $\gamma$ and the valuation bound on $u$ give
\[
 \frac1{p-1}<\alpha=w_*+v_p(u)
 \leq\frac p{p-1}-\frac3{4e}<\frac p{p-1}.
\]

In $L[[X]]$, differentiating $g^p=\exp(cX)$ gives $g'=(c/p)g$.
Since $g(0)=1$, comparison of coefficients gives
\[
 g(X)=\sum_{k\geq0}\frac{(c/p)^k}{k!}X^k
 \qquad\text{in }L[[X]].
\]
For $k=p^r$, the identity $v_p((p^r)!)=(p^r-1)/(p-1)$ gives
\begin{equation}\label{eq4-11}
 v_p\!\left(\frac{(c/p)^{p^r}}{(p^r)!}\right)
 =p^r\left(\alpha-\frac p{p-1}\right)+\frac1{p-1}
 \longrightarrow-\infty.
\end{equation}
This contradicts the requirement that the coefficients of an element of
$L\langle X\rangle$ tend to zero.
This proves \eqref{eq4-9}, so Lemma~\ref{lem4-7} shows that $B$ is a domain.
Since $L$ was any finite extension containing $\mu_N$,
$Z_{I,\lambda,n}$ is geometrically irreducible when $|I|<d$.
\end{proof}

Combining Propositions~\ref{prop4-5} and~\ref{prop4-8} determines the
possible supports at a fixed radius. For a $D_{H,n}$-subquotient $Q$ of
$M_n$, each irreducible component of its support over $D_{T,n}$ is an
irreducible component of some $Z_{I,\lambda,n}$, with
$I\subseteq I_{\mathrm{alg}}$ and $\lambda_i\in\Omega_i$. If $|I|<d$, it
is the whole $Z_{I,\lambda,n}$; if $|I|=d$, it is a point of
$Z_{I,\lambda,n}$. The number of possible subspaces $Z_{I,\lambda,n}$
with $|I|<d$ is bounded independently of $n$. The number of possible
points for $|I|=d$ may grow with $n$.

\subsection{Counting successive quotients using generic ranks}
\label{sec4-4}
For $1\leq s\leq d$ and $n\geq n_*$, let
\begin{equation}
 \begin{split}
 \mathcal Z_{s,n}
  &=\left\{Z_{I,\lambda,n}:
      \begin{array}{l}
      I\subseteq I_{\mathrm{alg}},\ |I|=d-s,\\
      \lambda\in\prod_{i\in I}\Omega_i,\quad
      Z_{I,\lambda,n}\ne\varnothing
      \end{array}\right\},\\
 N_s&=\sum_{\substack{I\subseteq I_{\mathrm{alg}}\\|I|=d-s}}
           \prod_{i\in I}|\Omega_i|.
 \end{split}
 \label{eq4-12}
\end{equation}
Thus $|\mathcal Z_{s,n}|\leq N_s$. The constants $N_s$ depend only on
$H$ and $M$. In particular, $N_{d}=1$ and
$\mathcal Z_{d,n}=\{\operatorname{Sp}D_{T,n}\}$.
Every member of $\mathcal Z_{s,n}$ is geometrically irreducible of
dimension $s$ by Proposition~\ref{prop4-8}.
For $Z=Z_{I,\lambda,n}\in\mathcal Z_{s,n}$, put
\[
 \mathfrak p_{Z,n}=(h_i-\lambda_i:i\in I)D_{T,n},\qquad
 K_{Z,n}=\operatorname{Frac}(D_{T,n}/\mathfrak p_{Z,n}).
\]
For a finite $D_{T,n}$-module $Q$ with
$\mathfrak p_{Z,n}Q_{\mathfrak p_{Z,n}}=0$, define
\begin{equation}
 \operatorname{rank}_{Z,n}(Q)
   =\dim_{K_{Z,n}}Q_{\mathfrak p_{Z,n}}.
 \label{eq4-13}
\end{equation}

\begin{lemma}\label{lem4-9}
Let $1\leq s<d$, $n\geq n_*$, and let $Q$ be a
$D_{H,n}$-subquotient of $M_n$ with $d_{D_{H,n}}(Q)\leq s$.
For every $Z\in\mathcal Z_{s,n}$, the rank in \eqref{eq4-13} is defined
and is at most $C_{\mathrm{fib}}$. These ranks are additive in short exact
sequences of such subquotients. If $d_{D_{H,n}}(Q)=s$, then
$\operatorname{rank}_{Z,n}(Q)>0$ for some $Z\in\mathcal Z_{s,n}$.
\end{lemma}
\begin{proof}
Propositions~\ref{prop4-5} and~\ref{prop4-8} identify the $s$-dimensional
support components with members of $\mathcal Z_{s,n}$.
Fix $Z=Z_{I,\lambda,n}$ and write $A=D_{T,n}$ and
$\mathfrak p=\mathfrak p_{Z,n}$.
If $Q_{\mathfrak p}\ne0$, then $\mathfrak p$ is minimal in
$\operatorname{Supp}_A(Q)$: a prime in $\operatorname{Supp}_A(Q)$ strictly contained in
$\mathfrak p$ would have dimension at least $s+1$.
Thus $Q_{\mathfrak p}$ has finite length over $A_{\mathfrak p}$ and
$\mathfrak p^N Q_{\mathfrak p}=0$ for some $N\geq1$.
For $i\in I$, Lemma~\ref{lem4-2} gives
$Q[(h_i-\lambda_i)^N]=Q[h_i-\lambda_i]$.
These are kernels of $A$-linear multiplication maps, so exactness of
localization gives
\[
 Q_{\mathfrak p}[(h_i-\lambda_i)^N]
 =Q_{\mathfrak p}[h_i-\lambda_i].
\]
Hence every $h_i-\lambda_i$ annihilates $Q_{\mathfrak p}$.
The same conclusion is immediate if $Q_{\mathfrak p}=0$, so in all cases
$\mathfrak pQ_{\mathfrak p}=0$ and $Q_{\mathfrak p}$ is a
$K_{Z,n}$-vector space.
Choose a rigid point $x\in Z$, with maximal ideal $\mathfrak m_x$ and
residue field $\kappa(x)$. By Nakayama's lemma, $Q_{\mathfrak m_x}$
can be generated by $\mu_x(Q)$ elements. These elements also generate
$Q_{\mathfrak p}$ after localization. Theorem~\ref{thm4-3} therefore gives
\[
 \operatorname{rank}_{Z,n}(Q)\leq\mu_x(Q)\leq C_{\mathrm{fib}}.
\]
Exactness of localization proves additivity. If $Q$ has canonical
dimension $s$, localization at the prime of any $s$-dimensional support
component is nonzero,
which proves positivity.
\end{proof}

Put $b_s=C_{\mathrm{fib}}N_s$ for $1\leq s<d$ and $b_d=m$.
Set $B_0=0$ and, for $1\leq s\leq d$,
\[
 B_s=b_s+(1+b_s)B_{s-1}=\prod_{j=1}^{s}(1+b_j)-1.
\]
These constants depend only on $H$ and $M$.

\begin{theorem}\label{thm4-10}
Fix $1\leq s\leq d$ and $n\geq n_*$. Let $Q$ be a
$D_{H,n}$-subquotient of $M_n$ with $d_{D_{H,n}}(Q)\leq s$.
In any finite chain of $D_{H,n}$-submodules of $Q$, at most $b_s$
successive quotients have canonical dimension $s$, and at most $B_s$
have positive canonical dimension. These bounds are uniform in $n$.
\end{theorem}
\begin{proof}
Write the chain in decreasing order, adjoining $Q$ and $0$ as endpoints
if necessary, and denote its successive quotients by $Q_j$.
For $s<d$, Lemma~\ref{lem4-9} gives
\[
 \begin{split}
 \#\{j:d_{D_{H,n}}(Q_j)=s\}
 &\leq\sum_j\sum_{Z\in\mathcal Z_{s,n}}\operatorname{rank}_{Z,n}(Q_j)\\
 &=\sum_{Z\in\mathcal Z_{s,n}}\operatorname{rank}_{Z,n}(Q)
 \leq C_{\mathrm{fib}}N_s.
 \end{split}
\]
For $s=d$, additivity of generic rank bounds the number of quotients
with canonical dimension $d$ by the generic rank of $Q$.
This is at most the rank of $M_n$, and hence at most $m$ by
Proposition~\ref{prop3-3}.

We prove the bound $B_s$ by induction on $s$; the case $s=0$ is immediate.
The successive quotients of canonical dimension $s$ divide the chain
into at most $b_s+1$ consecutive subchains. By exactness, the quotient of
the first module by the last in each subchain has canonical dimension
at most $s-1$. The induction hypothesis bounds the number of successive
quotients with positive canonical dimension in each subchain by $B_{s-1}$.
Thus the whole chain has at most $b_s+(b_s+1)B_{s-1}=B_s$ such quotients.
\end{proof}

\begin{corollary}\label{cor:fixed-radius-length}
For every fixed $n\geq n_*$, the module $M_n$ has finite length over $D_{H,n}$.
\end{corollary}
\begin{proof}
Use the canonical dimension function $d_{D_{H,n}}(-)$.
Theorem~\ref{thm4-10} bounds the number of successive quotients with
positive canonical dimension by $B_d$.

By Lemma~\ref{lem2-2}, a subquotient $Q$ of canonical dimension zero is
finite-dimensional over $E$. If the infinitesimal character is nonalgebraic,
$Q=0$. Otherwise write $Q=V^\vee\otimes_EW^\vee$. Lemma~\ref{lem2-4} and Corollary~\ref{cor4-4} give
\[
 c_H(W)\leq n,\qquad
 \dim_E Q=d_V\dim_EW\leq d_VC_{\mathrm{fib}}p^{dc_H(W)}
 \leq d_VC_{\mathrm{fib}}p^{dn}.
\]
The successive quotients with positive canonical dimension divide any
finite strictly decreasing chain into at most $B_d+1$ consecutive
subchains whose successive quotients have canonical dimension zero.
By exactness, the quotient of the first module by the last in each
nonempty subchain also has canonical dimension zero. If $\chi$ is algebraic, each subchain has length at most
$d_VC_{\mathrm{fib}}p^{dn}$, so every finite strictly decreasing chain
has length at most $B_d+(B_d+1)d_VC_{\mathrm{fib}}p^{dn}$.
If $\chi$ is nonalgebraic, the bound $B_d$ suffices.
Thus $M_n$ has finite length over $D_{H,n}$.
\end{proof}

\section{Canonical dimension zero and finite length}\label{sec5}
To complete the proof of Theorem~\ref{thm:finite-length}, we first bound the lengths of finite $G$-stable coadmissible chains in subquotients of canonical dimension zero. For $G=D_F^\times$, the character formulas of
Boyarchenko--Sabitova~\cite{bib1} relate conductors to coadjoint orbits
and yield lower bounds on representation dimensions over an algebraic
closure. Counting Galois conjugates strengthens these bounds over the
fixed coefficient field $E$. For $G=\GL_2(F)$, conjugation instead gives lower
bounds for the dimensions of invariant spaces in smooth representations.

\subsection{Coadjoint orbits and dimension bounds}\label{sec5-1}
Fix an algebraic closure $\overline E$ of $E$.
We first establish lower bounds for the dimensions of smooth irreducible
representations in terms of their conductors. Let $H$ be a
positive-dimensional uniform pro-$p$ group, write $\mathfrak h=\log H$,
and assume
\begin{equation}\label{eq5-1}
 [\mathfrak h,\mathfrak h]\subseteq p\mathfrak h\quad(p>2),\qquad
 [\mathfrak h,\mathfrak h]\subseteq8\mathfrak h\quad(p=2).
\end{equation}
Use $H_n$ and $c_H(W)$ as in Subsections~\ref{subsec:groups-valuations}
and~\ref{subsec:conductors-radii}, respectively.
Put $\epsilon_p=0$ for odd $p$ and $\epsilon_p=1$
for $p=2$, and set $L_p=p^{\epsilon_p}\mathfrak h$. We use
\[
 \mathfrak h^*=\operatorname{Hom}_{\Z_p}(\mathfrak h,\Z_p),\qquad
 L_p^\dagger=\operatorname{Hom}_{\mathrm{cont}}(L_p,\mu_{p^\infty}(\overline E)).
\]
Fix an isomorphism of abelian groups
$\psi:\Q_p/\Z_p\longrightarrow\mu_{p^\infty}(\overline E)$.
For $\lambda\in\operatorname{Hom}_{\Z_p}(L_p,\Q_p)$, define
\[
 \omega_\lambda(x)=\psi(\lambda(x)+\Z_p)\qquad(x\in L_p).
\]
The coadjoint and scalar actions are
\[
 \begin{aligned}
  (g\omega)(x)&=\omega(\operatorname{Ad}(g^{-1})x)&& (g\in H),\\
  (u\omega)(x)&=\omega(ux)=\omega(x)^u&&(u\in\Z_p^\times).
 \end{aligned}
\]
For $\omega\in L_p^\dagger$ define
\[
 \ell_p(\omega)=\min\{n\geq\epsilon_p:\omega(p^n\mathfrak h)=1\},\qquad
 \operatorname{Rad}_p(\omega)=\{x\in\mathfrak h:\omega([x,L_p])=1\}.
\]
Write $\chi_\tau(h)=\operatorname{tr}(\tau(h))$ for the character of a
finite-dimensional representation $\tau$ of $H$, and
$\chi_{\mathrm{cyc}}:\operatorname{Gal}(\overline E/E)\longrightarrow\Z_p^\times$
for the $p$-adic cyclotomic character.

\begin{lemma}\label{lem5-1}
Every smooth absolutely irreducible $\overline E$-representation $\tau$ of $H$
determines a finite coadjoint orbit $\Omega\subset L_p^\dagger$. For $\omega\in\Omega$,
\[
 \frac{\chi_\tau(\exp x)}{\dim\tau}
   =\frac1{|\Omega|}\sum_{\eta\in\Omega}\eta(x)\qquad(x\in L_p),
\]
\[
 \max\{\epsilon_p,c_H(\tau)\}=\ell_p(\omega),\qquad
 |\Omega|=[\mathfrak h:\operatorname{Rad}_p(\omega)],\qquad
 (\dim\tau)^2\geq|\Omega|.
\]
For odd $p$, $(\dim\tau)^2=|\Omega|$, and the orbit determines $\tau$.
Conjugation by $\sigma\in\operatorname{Gal}(\overline E/E)$ sends the orbit of
$\tau$ to $\chi_{\mathrm{cyc}}(\sigma)\Omega$.
\end{lemma}
\begin{proof}
The orbit and character formula are given by
\cite[Thm. 2.6 and (2.7), Thm. 2.8(a)--(b)]{bib1}.
Since the representations factor through finite groups, these formulas apply over
$\overline E$ as well.

Embed a cyclotomic field containing the values in the character formula into $\C$. Realize $\chi_\tau$ by a unitary representation over $\C$.
An average of roots of unity is one only when every term is one. The character
formula and the $H$-stability of $p^n\mathfrak h$ therefore give, for $n\geq\epsilon_p$,
\[
 H_n\subseteq\ker\tau\quad\Longleftrightarrow\quad n\geq\ell_p(\omega),
 \qquad\max\{\epsilon_p,c_H(\tau)\}=\ell_p(\omega).
\]
For $x\in\mathfrak h$, the bracket condition~\eqref{eq5-1} makes
\[
 u_x=\sum_{j\geq0}\frac{(-\operatorname{ad}x)^j}{(j+1)!}
     \in1+p\operatorname{End}_{\Z_p}(L_p)
\]
an automorphism, and $\exp(-\operatorname{ad}x)-1=-\operatorname{ad}x\,u_x$.
Hence
\[
 \operatorname{Stab}_H(\omega)=\exp\operatorname{Rad}_p(\omega).
\]
For sufficiently large $r$, the logarithm induces bijections of sets
\[
 \begin{aligned}
  H/H_r&\simeq\mathfrak h/p^r\mathfrak h,\\
  \operatorname{Stab}_H(\omega)/H_r
  &\simeq\operatorname{Rad}_p(\omega)/p^r\mathfrak h.
 \end{aligned}
\]
Comparing their cardinalities gives
$|\Omega|=[\mathfrak h:\operatorname{Rad}_p(\omega)]$.

For odd $p$, the equality $(\dim\tau)^2=|\Omega|$ is \cite[(2.7)]{bib1}.
For $p=2$, the logarithm identifies normalized Haar measure on $H_1$ with
additive Haar measure on $L_p=2\mathfrak h$, normalized to have total
mass one. Orthogonality of the distinct additive characters in $\Omega$
gives
\[
 \int_{L_p}\left|\frac1{|\Omega|}\sum_{\eta\in\Omega}\eta(x)\right|_{\C}^{2}\,dx
 =\frac1{|\Omega|}.
\]
Here $|\cdot|_{\C}$ is the complex absolute value under the chosen
embedding. The character formula and character orthogonality on $H_1$
therefore give
\[
 \dim_{\overline E}\operatorname{End}_{H_1}(\tau)
   =\frac{(\dim\tau)^2}{|\Omega|}\geq1.
\]
Finally, $\sigma(\eta(x))=\eta(x)^{\chi_{\mathrm{cyc}}(\sigma)}$;
uniqueness of the orbit in the character formula gives the Galois assertion.
\end{proof}

To obtain a lower bound on representation dimension in terms of the
conductor, put
\[
 r_H=\frac12\min_{0\ne\lambda\in\mathfrak h_{\Q_p}^*}
       \operatorname{rank}_{\Q_p}\bigl((x,y)\mapsto\lambda([x,y])\bigr),
 \qquad\mathfrak h_{\Q_p}=\Q_p\otimes_{\Z_p}\mathfrak h.
\]
Here $(x,y)\mapsto\lambda([x,y])$ is a skew-symmetric $\Q_p$-bilinear
form on $\mathfrak h_{\Q_p}$. Its rank is the rank of the associated linear
map $x\mapsto(y\mapsto\lambda([x,y]))$ from $\mathfrak h_{\Q_p}$ to
$\mathfrak h_{\Q_p}^*$.
Thus $2r_H$ is the minimum $\Q_p$-analytic dimension of the coadjoint orbit
$H\cdot\lambda\subseteq\mathfrak h_{\Q_p}^*$ as $\lambda$ ranges over the
nonzero $\Q_p$-linear functionals; this minimum may be zero.

\begin{proposition}\label{prop5-2}
There is $c_1=c_1(H)>0$ such that every smooth absolutely irreducible
$\overline E$-representation $\tau$ of $H$ satisfies
\[
 \dim\tau\geq c_1p^{r_Hc_H(\tau)}.
\]
The same estimate holds for smooth irreducible representations over $E$.
If $F/\Q_p$ is finite and $H$ is a sufficiently small compact open uniform
subgroup of $\SL_1(D_F)$ or $\SL_2(F)$, then $r_H=d$, where $d=[F:\Q_p]$.
\end{proposition}
\begin{proof}
If $r_H=0$, take $c_1=1$. Otherwise put $n=c_H(\tau)$ and first
assume $n>\epsilon_p$. Let $\Omega$ be the coadjoint orbit attached to $\tau$
by Lemma~\ref{lem5-1}. A character in $\Omega$ has the form $\omega_{p^{-n}\ell}$,
with $\ell\in\mathfrak h^*\setminus p\mathfrak h^*$. Define
\[
 \beta_\ell:\mathfrak h\longrightarrow\mathfrak h^*,\qquad
 \beta_\ell(x)(y)=\ell([x,y]).
\]
By Smith normal form, there are $\Z_p$-bases of $\mathfrak h$ and
$\mathfrak h^*$ in which $\beta_\ell$ has matrix
\[
 \operatorname{diag}\bigl(p^{s_1(\ell)},\ldots,p^{s_t(\ell)},0,\ldots,0\bigr),
 \qquad 0\leq s_1(\ell)\leq\cdots\leq s_t(\ell).
\]
The nonzero diagonal entries are its elementary divisors. Thus
$s_j(\ell)$ are their $p$-adic valuations and
$t=\operatorname{rank}_{\Q_p}\beta_\ell$.
Since $\ell\notin p\mathfrak h^*$, it extends to a nonzero $\Q_p$-linear
functional on $\mathfrak h_{\Q_p}$. By the definition of $r_H$, the form
$(x,y)\mapsto\ell([x,y])$ has rank at least $2r_H$. Hence
$t\geq2r_H$, so $\beta_\ell$ has at least $2r_H$ nonzero elementary divisors.
In fixed $\Z_p$-bases of $\mathfrak h$ and $\mathfrak h^*$, the Smith
normal form gives
\[
 \max_{\mathcal M}|\det\mathcal M|
   =p^{-\sum_{j=1}^{2r_H}s_j(\ell)},
\]
where $\mathcal M$ ranges over the $2r_H\times2r_H$ submatrices of
the matrix of $\beta_\ell$.
On the compact set $\mathfrak h^*\setminus p\mathfrak h^*$, this maximum
is a positive continuous function. It has a positive minimum, so
$\sum_{j=1}^{2r_H}s_j(\ell)\leq B$ for some $B=B(H)\geq0$.

The condition that $\exp x$ stabilize $\omega_{p^{-n}\ell}$ is
$\beta_\ell(x)\in p^{n-\epsilon_p}\mathfrak h^*$. Writing $x_j\in\Z_p$
for the coordinates of $x$ in the basis of $\mathfrak h$ used in the Smith normal form above, this requires
\[
 p^{s_j(\ell)}x_j\in p^{n-\epsilon_p}\Z_p
 \quad\Longleftrightarrow\quad
 x_j\in p^{\max\{n-\epsilon_p-s_j(\ell),0\}}\Z_p
 \qquad(1\leq j\leq t).
\]
The remaining coordinates are unrestricted. Multiplying these indices and applying Lemma~\ref{lem5-1} gives
\[
 \begin{aligned}
 \dim\tau&\geq|\Omega|^{1/2}
  =p^{\frac12\sum_{j=1}^t\max\{n-\epsilon_p-s_j(\ell),0\}}\\
 &\geq p^{r_H(n-\epsilon_p)-B/2}.
 \end{aligned}
\]
Set $c_1=p^{-r_H\epsilon_p-B/2}$. For $n\leq\epsilon_p$,
$c_1p^{r_Hn}\leq1\leq\dim\tau$, so the same bound holds.
After extending scalars to $\overline E$, the irreducible constituents of an irreducible $E$-representation
are Galois conjugate and have the same conductor. Applying the estimate to one
of them proves the bound over $E$.

For the two groups in the statement, $\mathfrak h_{\Q_p}$ is the trace-zero
subspace of $D_F$ or $M_2(F)$, respectively, and has $F$-dimension three.
The $F$-valued trace pairing is nondegenerate on the ambient algebra:
for $D_F$, every nonzero $x$ satisfies
$\operatorname{Trd}(xx^{-1})=2\ne0$; for $M_2(F)$, this is the usual
matrix trace pairing. The scalar line is orthogonal to the trace-zero subspace, and the trace
pairing takes the value $2\ne0$ on $(1,1)$. Hence its restriction to the
trace-zero subspace is nondegenerate. Together with the nondegeneracy of the
field trace pairing, this shows that
\[
 \langle x,y\rangle
   =\operatorname{Tr}_{F/\Q_p}\operatorname{Trd}(xy)
\]
identifies $\mathfrak h_{\Q_p}$ with its $\Q_p$-linear dual; in the $M_2(F)$ case, $\operatorname{Trd}$ denotes the usual matrix trace. Thus every nonzero
$\ell\in\mathfrak h_{\Q_p}^*$ has the form $\ell(x)=\langle z,x\rangle$
for a unique nonzero $z\in\mathfrak h_{\Q_p}$. Using the same formula for
$\beta_\ell$ on $\mathfrak h_{\Q_p}$, cyclicity of trace gives
\[
 \beta_\ell(x)(y)=\ell([x,y])=\langle[z,x],y\rangle.
\]
By nondegeneracy, $\ker\beta_\ell$ is therefore the centralizer of $z$ in
$\mathfrak h_{\Q_p}$.

To compute this centralizer in $\mathfrak{sl}_2(F)$, note that a nonzero
trace-zero $z$ is nonscalar. Choose $v$ such that $v,zv$ form a basis of
$F^2$. A matrix commuting with $z$ is determined by its value on $v$,
so its centralizer in $M_2(F)$ is $F\cdot1+Fz$. Since
$\operatorname{tr}(a\,1+bz)=2a$, the trace-zero centralizer is exactly $Fz$.
This argument includes nonzero nilpotent $z$. For $D_F$, the same
calculation after splitting shows that the trace-zero centralizer also has
$F$-dimension one, since dimensions of kernels are preserved by field
extension. Hence in both cases
\[
 \operatorname{rank}_{\Q_p}\beta_\ell
   =\dim_{\Q_p}\mathfrak h_{\Q_p}-\dim_{\Q_p}\ker\beta_\ell
   =3d-d=2d.
\]
Taking the minimum over nonzero $\ell$ gives $r_H=d$.
\end{proof}

\subsection{Dimension bounds for compact subgroups of \texorpdfstring{$\SL_1(D_F)$}{SL1(D_F)} over \texorpdfstring{$E$}{E}}\label{sec5-2}
\begin{proposition}\label{prop5-3}
Let $H\subset\SL_1(D_F)$ be compact open and uniform,
satisfying~\eqref{eq5-1}. There is $c=c(H)>0$ such that every
smooth irreducible $E$-representation $W$ of $H$ satisfies
\[
 \dim_EW\geq cp^{(d+1)c_H(W)}.
\]
\end{proposition}
\begin{proof}
Choose an irreducible constituent $\tau$ of $\overline E\otimes_EW$. Its
conjugates form one Galois orbit and have common conductor $n=c_H(W)$.
It suffices to treat $n>\epsilon_p$; the remaining cases are covered by decreasing the constant $c$. Proposition~\ref{prop5-2} gives
\[
 \dim_{\overline E}\tau\geq c_1p^{dn}.
\]
To bound the number of Galois conjugates, choose a character $\omega_{p^{-n}\ell}$
in the orbit attached to $\tau$ by Lemma~\ref{lem5-1}, with
$\ell\in\mathfrak h^*\setminus p\mathfrak h^*$.
Let $z_\ell\in\mathfrak{sl}_1(D_F)$ correspond to $\ell$ under the trace
pairing used in the proof of Proposition~\ref{prop5-2}.
Choose an integer $a=a(H)\geq0$ such that the forms
\[
 Q(\ell)=p^a\operatorname{Nrd}(z_\ell),\qquad
 B_Q(\ell,\ell')=Q(\ell+\ell')-Q(\ell)-Q(\ell')
\]
take values in $\mathcal O_F$ on $\mathfrak h^*$ and
$\mathfrak h^*\times\mathfrak h^*$, respectively. Since $D_F$ is a division algebra, $Q(\ell)\ne0$ on the compact set
$\mathfrak h^*\setminus p\mathfrak h^*$. Consequently, there is a nonnegative integer
$B_{\mathrm{Nrd}}=B_{\mathrm{Nrd}}(H)$ such that
$v_p(Q(\ell))\leq B_{\mathrm{Nrd}}$ for all
$\ell\in\mathfrak h^*\setminus p\mathfrak h^*$.

If $\sigma\in\operatorname{Gal}(\overline E/E)$ fixes the isomorphism class of
$\tau$, it preserves its associated coadjoint orbit. Putting
$u_\sigma=\chi_{\mathrm{cyc}}(\sigma)$, we obtain
\[
 u_\sigma\ell-\operatorname{Ad}^*(g)\ell\in p^{n-\epsilon_p}\mathfrak h^*
 \quad\text{for some }g\in H.
\]
Here $(\operatorname{Ad}^*(g)\ell)(x)=\ell(\operatorname{Ad}(g^{-1})x)$.
The trace pairing identifies this action with conjugation on $z_\ell$, so
$Q(\operatorname{Ad}^*(g)\ell)=Q(\ell)$.
Put $\delta=u_\sigma\ell-\operatorname{Ad}^*(g)\ell$. The quadratic
expansion gives
\[
 Q(u_\sigma\ell)-Q(\operatorname{Ad}^*(g)\ell)
 =B_Q(\operatorname{Ad}^*(g)\ell,\delta)+Q(\delta).
\]
Since $B_Q$ and $Q$ are integral and
$\delta\in p^{n-\epsilon_p}\mathfrak h^*$, the two terms on the right belong to
$p^{n-\epsilon_p}\mathcal O_F$ and
$p^{2(n-\epsilon_p)}\mathcal O_F$, respectively.
Using $Q(u_\sigma\ell)=u_\sigma^2Q(\ell)$, we obtain
\[
 (u_\sigma^2-1)Q(\ell)\in p^{n-\epsilon_p}\mathcal O_F,\qquad
 v_p(u_\sigma^2-1)\geq n-\epsilon_p-B_{\mathrm{Nrd}}.
\]
Put $\Gamma_E=\operatorname{Gal}(\overline E/E)$ and
$b_n=n-\epsilon_p-B_{\mathrm{Nrd}}$, and first suppose $b_n\geq1$.
Let $\Gamma_{b_n}$ be the image of $\Gamma_E$ under the cyclotomic character
modulo $p^{b_n}$, and let $S_{b_n}$ be the image of
$\operatorname{Stab}_{\Gamma_E}(\tau)$. Since the cyclotomic image of
$\Gamma_E$ is open in $\Z_p^\times$, there is a constant $c_E=c_E(E)>0$
such that $|\Gamma_{b_n}|\geq c_Ep^{b_n}$. Every element of $S_{b_n}$ satisfies
$u^2=1$ modulo $p^{b_n}$, so $|S_{b_n}|\leq4$. The induced map of coset spaces
\[
 \Gamma_E/\operatorname{Stab}_{\Gamma_E}(\tau)
 \longrightarrow\Gamma_{b_n}/S_{b_n}
\]
is surjective. Hence
\[
 \begin{aligned}
 {}[\Gamma_E:\operatorname{Stab}_{\Gamma_E}(\tau)]
 &\geq[\Gamma_{b_n}:S_{b_n}]\geq\frac{|\Gamma_{b_n}|}{4}\geq\frac{c_Ep^{b_n}}{4}\\
 &=\frac{c_Ep^{-\epsilon_p-B_{\mathrm{Nrd}}}}4\,p^n.
 \end{aligned}
\]
Take $c_2=c_Ep^{-\epsilon_p-B_{\mathrm{Nrd}}}/4$ and decrease it to
cover the remaining finitely many values of $n$. The distinct Galois
conjugates of $\tau$ all occur in $\overline E\otimes_EW$ and have the same
dimension. Therefore,
\[
 \dim_EW\geq[\Gamma_E:\operatorname{Stab}_{\Gamma_E}(\tau)]
       \dim_{\overline E}\tau\geq c_1c_2p^{(d+1)n}.
\]
\end{proof}
The extra factor $p^{c_H(W)}$ comes from the lower bound on the number
of Galois conjugates.

\subsection{Growth of invariant spaces for \texorpdfstring{$\SL_2(F)$}{SL2(F)}}\label{sec5-3}
\begin{proposition}\label{prop5-4}
Fix a sufficiently small compact open uniform $H\subset\SL_2(F)$.
There is $\kappa_H=\kappa_H(H)>0$ such that every infinite-dimensional smooth
irreducible admissible $E$-representation $\tau$ of $\SL_2(F)$ admits an integer
$n_0=n_0(H,\tau)$ for which
\[
 \dim_E\tau^{H_n}\geq\kappa_Hp^{dn}\qquad(n\geq n_0).
\]
\end{proposition}
\begin{proof}
Put $K_r=\exp(p^r\mathfrak{sl}_2(\mathcal O_F))$ for sufficiently large integers
$r$, and first take $H=K_s$ with $s$ sufficiently large. The finite-dimensional
invariant spaces form an increasing sequence
\[
 \mathcal V_0\subseteq\mathcal V_1\subseteq\cdots,\qquad
 \mathcal V_n=\tau^{K_{s+n}}.
\]
Let $K=\SL_2(\mathcal O_F)$ and $a=\operatorname{diag}(\pi_F,\pi_F^{-1})$,
with $\pi_F$ a uniformizer of $F$. Matrix conjugation gives
\[
 a^{\pm1}K_{s+n+2}a^{\mp1}\subseteq K_{s+n},\qquad
 \tau(a^{\pm1})\mathcal V_n\subseteq\mathcal V_{n+2}.
\]
Here the subscript measures powers of $p$. Conjugation by $a$ multiplies
the upper and lower off-diagonal entries by $\pi_F^2$ and $\pi_F^{-2}$,
respectively. Since $v_F(p^2/\pi_F^2)=2e-2\geq0$, the displayed subgroup inclusions hold for both $a$ and $a^{-1}$.
The group generated by $K$ and $a$ is $\SL_2(F)$ by the
elementary divisor decomposition. Since $K_{s+n}$ is normal in $K$, the space
$\mathcal V_n$ is $K$-stable. If $\mathcal V_n=\mathcal V_{n+2}\ne0$, then
\[
 \tau(a^{\pm1})\mathcal V_n=\mathcal V_n,\qquad
 \SL_2(F)\cdot\mathcal V_n=\mathcal V_n.
\]
Thus $\mathcal V_n$ is a nonzero finite-dimensional $\SL_2(F)$-subrepresentation
of $\tau$, contradicting irreducibility and infinite-dimensionality. Since $\tau$
is smooth and nonzero, $\mathcal V_n\ne0$ for large $n$. Consequently
\[
 (\mathcal V_n)^{K_{s+n-2}}=\mathcal V_{n-2}\subsetneq\mathcal V_n
 \qquad(n\gg0).
\]
After extending scalars to $\overline E$, the semisimple
$K_s/K_{s+n}$-representation $\overline E\otimes_E\mathcal V_n$ has
an irreducible constituent whose conductor relative to $K_s$ is
$n-1$ or $n$. Otherwise $K_{s+n-2}$ would act trivially on
$\mathcal V_n$.
Let $c_1=c_1(K_s)>0$ be the constant in Proposition~\ref{prop5-2}.
It gives $\dim_E\mathcal V_n\geq c_1p^{d(n-1)}$.
For a general $H$, keep $K_s$ fixed. Since $\log H$ and $p^s\mathfrak{sl}_2(\mathcal O_F)$ are commensurable, there is $a_0=a_0(H)$ such that $H_n\subseteq K_{s+n-a_0}$ for all sufficiently large $n$. Thus
$\tau^{K_{s+n-a_0}}\subseteq\tau^{H_n}$, and
\[
 \dim_E\tau^{H_n}\geq c_1p^{d(n-a_0-1)}\qquad(n\gg0).
\]
Take $\kappa_H=c_1p^{-d(a_0+1)}$.
\end{proof}

\subsection{Subquotients of canonical dimension zero}\label{sec5-4}
We use the notation of Subsection~\ref{subsec:conductors-radii},
with $H$ satisfying~\eqref{eq5-1}.
Lemma~\ref{lem2-5} shows that if $\chi$ is nonalgebraic, every coadmissible
subquotient of $M_{\an}$ of canonical dimension zero is zero.
In the algebraic case, choose $n_*=n_*(H,M)\geq n_{\mathrm{res}}$ and
$C_{\mathrm{fd}}=C_{\mathrm{fd}}(H,M)>0$ such that, for every $n\geq n_*$,
every subquotient $Q=V^\vee\otimes_EW^\vee$ of $M_n$ of canonical
dimension zero satisfies
\begin{equation}\label{eq5-2}
 \dim_EW\leq C_{\mathrm{fd}}p^{dc_H(W)}.
\end{equation}
Propositions~\ref{prop5-5} and~\ref{prop5-6} below use~\eqref{eq5-2}
as a hypothesis in the algebraic case.
For $p>2$, enlarge $n_*$ from Proposition~\ref{prop3-3} if necessary;
Corollary~\ref{cor4-4} gives~\eqref{eq5-2} with
$C_{\mathrm{fd}}=C_{\mathrm{fib}}=2^{d}m$.

\paragraph*{The case $G=D_F^\times$.}
\begin{proposition}\label{prop5-5}
Let $G=D_F^\times$. If $\chi$ is algebraic, there are integers
$B=B(H,M)$ and $n_\dagger=n_\dagger(H,M)$ such that every coadmissible
subquotient $Q$ of $M_{\an}$ of canonical dimension zero satisfies
\[
 \dim_EQ\leq B,\qquad Q\xrightarrow{\sim}Q_n
 \quad\text{for every }n\geq n_\dagger,
\]
where the map is canonical. If $\chi$ is nonalgebraic, every such $Q$ is zero.
\end{proposition}
\begin{proof}
The case $Q=0$ is immediate. For $Q\ne0$, Lemma~\ref{lem2-5} gives
\[
 Q'_b\simeq V\otimes_E\rho,\qquad
 Q_n\simeq V^\vee\otimes_EW_n^\vee,\qquad W_n=\rho^{H_n},
\]
with $\rho$ smooth admissible. For $n\geq n_*$ with $W_n\ne0$, the action
through the finite group $H/H_n$ makes $W_n$ semisimple.
Let $c=c(H)>0$ be the constant in Proposition~\ref{prop5-3}.
Choose an irreducible $E$-summand of maximal conductor
$\nu_n=c_H(W_n)$ and apply Proposition~\ref{prop5-3} over the fixed
field $E$. Together with~\eqref{eq5-2}, this gives
\[
 cp^{(d+1)\nu_n}\leq\dim_EW_n\leq C_{\mathrm{fd}}p^{d\nu_n}.
\]
Thus $p^{\nu_n}\leq C_{\mathrm{fd}}/c$, so
$\nu_n\leq c_{\mathrm{bd}}=c_{\mathrm{bd}}(H,M)$ for an
integer $c_{\mathrm{bd}}\geq1$, uniformly in $n$ and $Q$.

Since $\rho=\bigcup_{n\geq n_*}W_n$, the group $H_{c_{\mathrm{bd}}}$ acts
trivially on $\rho$. Admissibility therefore makes
$\rho=\rho^{H_{c_{\mathrm{bd}}}}$ finite-dimensional. For
$n\geq n_\dagger:=\max\{n_*,c_{\mathrm{bd}}\}$, one has $W_n=\rho$.
The transition maps are dual to the inclusions
$V\otimes_EW_n\subseteq V\otimes_EW_{n+1}$, so the inverse
system is stationary from $n_\dagger$ onward and the canonical map
$Q=\varprojlim_nQ_n\to Q_n$ is an isomorphism for $n\geq n_\dagger$.
Applying~\eqref{eq5-2} at such a radius gives
\[
 \dim_E\rho\leq C_{\mathrm{fd}}p^{dc_{\mathrm{bd}}}.
\]
Since $\dim_EQ=d_V\dim_E\rho$, we may take
\[
 B=\left\lceil d_VC_{\mathrm{fd}}p^{dc_{\mathrm{bd}}}\right\rceil.
\]
\end{proof}

\paragraph*{The case $G=\GL_2(F)$.}
The estimate in Proposition~\ref{prop5-3} uses the anisotropy of the reduced
norm on $D_F$. For $\GL_2(F)$, we instead use the lower bound for
$\dim_E\tau^{H_n}$ in Proposition~\ref{prop5-4}.

\begin{proposition}\label{prop5-6}
Let $G=\GL_2(F)$. There is an integer $b_0=b_0(H,M)\geq0$ such that
every $G$-stable coadmissible subquotient $Q$ of $M_{\an}$ of canonical
dimension zero has finite length at most $b_0$, computed using
$G$-stable coadmissible submodules. If $\chi$ is nonalgebraic, take $b_0=0$.
\end{proposition}
\begin{proof}
The case $Q=0$ is immediate. Suppose $Q\ne0$. By Lemma~\ref{lem2-2},
every radius module $Q_n$ is finite-dimensional.
Take any finite strictly decreasing chain of $G$-stable coadmissible
submodules of $Q$. Adjoin $Q$ and $0$ as endpoints if necessary, and write
\[
 Q=Q^0\supsetneq Q^1\supsetneq\cdots\supsetneq Q^\ell=0,\qquad
 S_i=Q^{i-1}/Q^i\quad(1\leq i\leq\ell).
\]
Regard this as a $G_1$-stable chain. Lemma~\ref{lem2-5}, with $\Gamma=G_1$,
gives $Q'_b=V\otimes_E\rho$,
where $\rho$ is a smooth admissible representation of $G_1$.
The functor $\operatorname{Hom}_{\mathfrak g_E}(V,-)$ is exact on these
$\mathfrak g_E$-modules, which are direct sums of copies of $V$. Applied to
the dual chain, it gives a filtration of $\rho$ in reverse order, with
nonzero successive quotients
\[
 \rho_i=\operatorname{Hom}_{\mathfrak g_E}(V,(S_i)'_b),\qquad
 (S_i)'_b\simeq V\otimes_E\rho_i.
\]
For every $n\geq n_*$, exactness of taking $H_n$-invariants and Lemma~\ref{lem2-5} give
\begin{equation}\label{eq5-3}
 \sum_{i=1}^\ell\dim_E\rho_i^{H_n}
   =\dim_E\rho^{H_n}\leq C_{\mathrm{fd}}p^{dn},
\end{equation}
where the inequality uses~\eqref{eq5-2} and $c_H(\rho^{H_n})\leq n$.
For each $i$, choose an irreducible subquotient $\tau_i$ of $\rho_i$.
Such a subquotient exists: take a nonzero cyclic subrepresentation of
$\rho_i$ and quotient by a maximal proper subrepresentation.
The representation $\tau_i$ is smooth and admissible, since taking
invariants under a compact open subgroup is exact for smooth
representations over $E$. Set
\[
 I_{\mathrm{fin}}=\{i:\dim_E\tau_i<\infty\},\qquad
 I_\infty=\{i:\dim_E\tau_i=\infty\}.
\]
For $i\in I_{\mathrm{fin}}$, the representation $\tau_i$ is trivial on $G_1$.
Indeed, its open normal kernel contains upper unipotents with entries in a
sufficiently small ideal of $\mathcal O_F$. Conjugation by all powers of
$\operatorname{diag}(\pi_F,\pi_F^{-1})$ gives every upper unipotent; the Weyl
element gives every lower unipotent. These root groups generate $G_1$.
Exactness of taking $H_{n_*}$-invariants gives
$0<\dim_E\tau_i^{H_{n_*}}\leq\dim_E\rho_i^{H_{n_*}}$. By~\eqref{eq5-3},
\[
 |I_{\mathrm{fin}}|\leq\sum_{i\in I_{\mathrm{fin}}}\dim_E\rho_i^{H_{n_*}}
 \leq C_{\mathrm{fd}}p^{dn_*}.
\]
Proposition~\ref{prop5-4} gives $\kappa_H=\kappa_H(H)>0$ and a threshold
$n_0(H,\tau_i)$ for each $i\in I_\infty$. Since $I_\infty$ is finite, choose
\[
 n\geq\max\bigl(\{n_*\}\cup\{n_0(H,\tau_i):i\in I_\infty\}\bigr).
\]
The index $n$ may depend on the chain, but $C_{\mathrm{fd}}/\kappa_H$ does
not. Exactness of taking $H_n$-invariants and~\eqref{eq5-3} give
\[
 \begin{aligned}
 |I_\infty|\kappa_Hp^{dn}
  &\leq\sum_{i\in I_\infty}\dim_E\tau_i^{H_n}\\
  &\leq\sum_{i=1}^\ell\dim_E\rho_i^{H_n}
     \leq C_{\mathrm{fd}}p^{dn}.
 \end{aligned}
\]
Since $\ell=|I_{\mathrm{fin}}|+|I_\infty|$, we may take
\[
 b_0=\left\lceil C_{\mathrm{fd}}p^{dn_*}
       +\frac{C_{\mathrm{fd}}}{\kappa_H}\right\rceil.
\]
Thus $Q$ has finite length at most $b_0$.
\end{proof}

\subsection{Proof of finite length}\label{sec5-5}
We first record the passage from coadmissible chains to finite length of
representations.

\begin{lemma}\label{lem5-8}
In the notation of Subsection~\ref{subsec:conductors-radii}, suppose finite strictly
decreasing chains of $G$-stable coadmissible submodules of $M_{\an}$ have
a common length bound. Then $\Pi^{\an}$ and
$\Pi$ have finite length in the categories of admissible locally analytic
and admissible Banach representations of $G$, respectively.
\end{lemma}
\begin{proof}
Given a finite strictly decreasing chain of admissible locally analytic
$G$-subrepresentations of $\Pi^{\an}$, take annihilators and reverse the
order. Strong duality \cite[Thm. 6.3 and Prop. 6.4]{bib6} gives a finite
strictly decreasing chain of $G$-stable coadmissible submodules of $M_{\an}$
of the same length.
Thus $\Pi^{\an}$ has finite length. Choose a compact open subgroup
$K\subseteq G$. Locally analytic vectors can be computed on $K$.
For a strict inclusion $A\subsetneq B$ of closed $G$-stable Banach
subspaces, locally analytic vectors are dense in the nonzero representation
$B/A$, so $(B/A)^{\an}\ne0$. Exactness of taking locally analytic vectors
gives $(B/A)^{\an}\simeq B^{\an}/A^{\an}$ \cite[Thm. 7.1]{bib6}.
Thus every finite strictly decreasing chain of closed $G$-stable Banach
subspaces gives a finite strictly decreasing chain of admissible locally
analytic $G$-subrepresentations.
The bound just proved therefore also gives finite length of $\Pi$.
\end{proof}

For a nonzero coadmissible $D_H$-subquotient $Q$ of $M_{\an}$,
$d_{D_H}(Q)=0$ if and only if every $Q_n$ is finite-dimensional over $E$,
by Lemma~\ref{lem2-2}.
For $p>2$, use the subgroup and torus of Section~\ref{sec:analytic-torus};
Proposition~\ref{prop3-6} identifies this condition with $\delta_T(Q)=0$.

The proof of Corollary~\ref{cor:fixed-radius-length} gives a bound
depending on $n$ for the $E$-dimensions of subquotients of $M_n$ of
canonical dimension zero. Propositions~\ref{prop5-5} and~\ref{prop5-6}
give the uniform bounds for $G$-stable coadmissible chains used in
Theorem~\ref{thm:finite-length}.

\begin{proof}[Proof of Theorem \ref{thm:finite-length} for $p>2$]
Use the subgroup and torus of Section~\ref{sec:analytic-torus}, after applying Lemma~\ref{lem2-3} and making
the fixed finite extension of $E$. Work with $G$-stable coadmissible
subquotients of $M_{\an}$ and the exact dimension function $\delta_T$ from
Subsection~\ref{subsec:torus-finiteness}.
Take a finite strictly decreasing chain in such a subquotient $Q$,
adjoining $Q$ and $0$ as endpoints if necessary.
For each nonzero successive quotient $S$, Subsection~\ref{subsec:torus-finiteness} gives
\[
 \dim\operatorname{Supp}_{D_{T,n}}S_n=\delta_T(S)
\]
for all sufficiently large $n$. Since the chain is finite, choose one
$n\geq n_*$ for which this equality holds for every such $S$.
Exact base change gives a finite strictly decreasing chain of
$D_{H,n}$-submodules in the subquotient $Q_n$ of $M_n$.
Theorem~\ref{thm4-10} bounds the number of successive quotients with
positive canonical dimension by $B_d$, independently of the radius and
the chain.

If $\chi$ is algebraic, take $b_0=B$ from Proposition~\ref{prop5-5}
for $G=D_F^\times$, and take the bound $b_0$ in Proposition~\ref{prop5-6}
for $G=\GL_2(F)$. If $\chi$ is nonalgebraic, take $b_0=0$.
This bounds the lengths of finite strictly decreasing chains of
$G$-stable coadmissible submodules in subquotients $S$ with $\delta_T(S)=0$.
The successive quotients of positive canonical dimension divide the chain into at most
$B_d+1$ consecutive subchains. For each nonempty subchain, let $S$ be the
quotient of its first module by its last. Exactness of $\delta_T$ gives
$\delta_T(S)=0$, so the subchain has length at most $b_0$.
Thus the original chain has length at most $B_d+(B_d+1)b_0$.
Lemma~\ref{lem5-8} gives the assertions for $\Pi^{\an}$ and $\Pi$.
\end{proof}

\appendix
\section{A simpler proof for \texorpdfstring{$F=\Q_p$}{F=Qp}}\label{appA}
We give a simpler proof of Theorem~\ref{thm:finite-length} for $F=\Q_p$.
The main simplification concerns the estimate for subquotients of
canonical dimension zero at a fixed radius.
Proposition~\ref{propA-3} directly bounds the dimensions of their smooth
factors in terms of the conductor, using an associated graded module
of Krull dimension at most one.
Ranks over a polynomial ring in one variable bound the number of successive
quotients of canonical dimension one. Combining these estimates with the bounds in canonical dimension
zero from Subsection~\ref{sec5-4} proves finite length for every prime.

\subsection{Estimates for graded modules}\label{app:graded-estimates}
Let $k$ be a field, let $r\geq1$, and give the variables of
$k[X_1,\ldots,X_r]$ positive rational degrees. Let $N$ be a finite graded
$k[X_1,\ldots,X_r]$-module of Krull dimension at most one. Put
\[
 \mathfrak m=(X_1,\ldots,X_r),\qquad
 N_{\mathrm{tor}}=\{x\in N:\mathfrak m^ax=0\text{ for some }a\geq1\}.
\]
The module $N_{\mathrm{tor}}$ has finite length.

Multiply the degrees by a common denominator, writing
$\deg X_i=a_i\in\Z_{>0}$, and put $\ell=\operatorname{lcm}(a_1,\ldots,a_r)$.

\begin{lemma}\label{lemA-1}
There are a finite extension $k'/k$ and a nonzero homogeneous form
$\phi\in k'[X_1,\ldots,X_r]$ of degree $\ell$ such that, writing
\[
 N'=k'\otimes_kN,\qquad N'_{\mathrm{tor}}=k'\otimes_kN_{\mathrm{tor}},
\]
the module $N'$ is finite over $k'[\phi]$ and $N'/N'_{\mathrm{tor}}$ is
finite free over $k'[\phi]$.
\end{lemma}
\begin{proof}
If $N/N_{\mathrm{tor}}=0$, take $k'=k$ and any nonzero degree-$\ell$ form;
then $N=N_{\mathrm{tor}}$ is finite-dimensional. Suppose that
$N/N_{\mathrm{tor}}\ne0$ and put $V_\ell=k[X_1,\ldots,X_r]_\ell$, the
space of homogeneous forms of weighted degree $\ell$. Fix an algebraic
closure $\overline k$ of $k$ and write
$\overline V_\ell=\overline k\otimes_kV_\ell$ and
$\overline N=\overline k\otimes_k(N/N_{\mathrm{tor}})$.
Flat base change preserves the absence of $\mathfrak m$-torsion.
Every associated prime $\mathfrak p$ of $\overline N$ therefore meets
$\overline V_\ell$ in a proper linear subspace: otherwise $\mathfrak p$
would contain $X_i^{\ell/a_i}$ for every $i$, hence $\mathfrak p=\mathfrak m$,
contradicting the absence of $\mathfrak m$-torsion. Since $\overline k$ is
infinite, finitely many proper linear subspaces cannot cover
$\overline V_\ell$, so choose
\[
 \phi\in\overline V_\ell\setminus
 \bigcup_{\mathfrak p\in\operatorname{Ass}_{\overline k[X_1,\ldots,X_r]}(\overline N)}
  (\mathfrak p\cap\overline V_\ell).
\]
Its coefficients lie in a finite extension $k'/k$. Then $\phi$ is regular on $N'/N'_{\mathrm{tor}}$. In particular, $\phi$ does not vanish identically on any one-dimensional component of
$\operatorname{Supp}_{k'[X_1,\ldots,X_r]}(N'/N'_{\mathrm{tor}})$.
Since this support has Krull dimension at most one, the quotient by $\phi$
has dimension zero. The quotient is graded, so its support satisfies
\[
 \operatorname{Supp}_{k'[X_1,\ldots,X_r]}
 \bigl((N'/N'_{\mathrm{tor}})/\phi(N'/N'_{\mathrm{tor}})\bigr)
 \subseteq\{(X_1,\ldots,X_r)\}.
\]
This quotient and $N'_{\mathrm{tor}}$ are finite-dimensional.
Hence $N'/\phi N'$ is finite-dimensional.
Positive grading shows that $N'$ is finite over $k'[\phi]$: homogeneous
lifts of a basis modulo $\phi$ generate, by induction on degree.
Suppose $P(\phi)x=0$ for a nonzero polynomial $P(t)\in k'[t]$ and a
nonzero $x\in N'/N'_{\mathrm{tor}}$. Let $ct^j$ be the leading term of
$P(t)$ and $x_{\max}$ the highest nonzero homogeneous component of $x$.
Since $\deg\phi>0$, the highest homogeneous component of $P(\phi)x$
is $c\phi^jx_{\max}$, which is nonzero by regularity of $\phi$.
This contradiction shows that $N'/N'_{\mathrm{tor}}$ is torsion-free.
Being finite, it is free over the principal ideal domain $k'[\phi]$.
\end{proof}

Fix $k'$ and $\phi$ as in Lemma~\ref{lemA-1}. For a graded subquotient
$Y$ of $N$, set
\[
 e_{\mathrm{res}}(Y)=\operatorname{rank}_{k'[\phi]}(k'\otimes_kY).
\]
This integer is additive in short exact sequences and is positive exactly
when $Y$ has Krull dimension one. Put
\[
 \begin{aligned}
 m_{\mathrm{res}}=m_{\mathrm{res}}(N;k',\phi)&=e_{\mathrm{res}}(N)\\
 &=\dim_{k'}\bigl((N'/N'_{\mathrm{tor}})/\phi(N'/N'_{\mathrm{tor}})\bigr).
 \end{aligned}
\]
If $N/N_{\mathrm{tor}}=0$, then $m_{\mathrm{res}}=0$.

Put $C_N=\dim_k N_{\mathrm{tor}}+r\,m_{\mathrm{res}}$.

\begin{lemma}\label{lemA-2}
For every graded submodule $C\subseteq N$ and every integer $q\geq1$,
\[
 \dim_k C/(X_1^q,\ldots,X_r^q)C
 \leq\dim_k N_{\mathrm{tor}}+rq\,m_{\mathrm{res}}\leq C_Nq.
\]
\end{lemma}
\begin{proof}
For a finite-dimensional $k$-space $U$,
$\dim_{k'}(k'\otimes_kU)=\dim_kU$. We may therefore replace $k$ by $k'$ and
suppress primes. Put $C_0=C/(C\cap N_{\mathrm{tor}})$.
As a $k[\phi]$-submodule of $N/N_{\mathrm{tor}}$, the module $C_0$ is
finite free of rank at most $m_{\mathrm{res}}$.

Since $\phi\in\mathfrak m$,
\[
 \phi^{rq}\in\mathfrak m^{rq}\subseteq(X_1^q,\ldots,X_r^q).
\]
The last inclusion follows by considering ordinary total degree. Consequently
\[
 \begin{aligned}
 \dim C/(X_1^q,\ldots,X_r^q)C
 &\leq\dim C/\phi^{rq}C\\
 &\leq\dim_k N_{\mathrm{tor}}+\dim C_0/\phi^{rq}C_0\\
 &=\dim_k N_{\mathrm{tor}}+rq\operatorname{rank}_{k[\phi]}C_0\\
 &\leq\dim_k N_{\mathrm{tor}}+rq\,m_{\mathrm{res}}.
 \end{aligned}
\]
\end{proof}

\begin{lemma}\label{lemA-4}
Let $R=k[\eps,X_1,\ldots,X_r]$ have a positive rational grading, and let
$P$ be a finite graded $R$-module on which $\eps$ is regular.
Suppose $P/\eps P\simeq N$, with $N$ and $m_{\mathrm{res}}$ as above.
In any finite chain of graded $R[1/\eps]$-submodules of $P[1/\eps]$,
at most $m_{\mathrm{res}}$ successive quotients have Krull dimension two.
\end{lemma}
\begin{proof}
Write the chain as $V_0\supseteq\cdots\supseteq V_s$, and put
$P_i=P\cap V_i$. Clearing denominators gives $P_i[1/\eps]=V_i$.
If $x\in P$ and $\eps x\in P_i$, then
$x\in V_i$ because $\eps$ is invertible on $V_i$, so $x\in P_i$.
Thus $P_i$ is $\eps$-saturated in $P$, so $P/P_i$ and $Q_i=P_{i-1}/P_i$ are $\eps$-torsion-free.
Exactness of localization identifies
\[
 Q_i[1/\eps]\simeq V_{i-1}/V_i.
\]
Reduction gives exact sequences
\[
 0\longrightarrow P_i/\eps P_i\longrightarrow P_{i-1}/\eps P_{i-1}
 \longrightarrow Q_i/\eps Q_i\longrightarrow0.
\]
The modules $P_i/\eps P_i$ embed in $N$, so the displayed sequences
describe a chain in $N$ and its successive quotients. For $Q_i[1/\eps]$
of dimension two, regularity of $\eps$ on $Q_i$ and the positive grading give
\[
 2=\operatorname{Kdim}Q_i[1/\eps]\leq\operatorname{Kdim}Q_i
 =1+\operatorname{Kdim}(Q_i/\eps Q_i)\leq2.
\]
Hence $\operatorname{Kdim}(Q_i/\eps Q_i)=1$ and
$e_{\mathrm{res}}(Q_i/\eps Q_i)>0$. Additivity of $e_{\mathrm{res}}$ along the chain $(P_i/\eps P_i)_i$ gives
\[
 \#\{i:\operatorname{Kdim}Q_i[1/\eps]=2\}
 \leq\sum_i e_{\mathrm{res}}(Q_i/\eps Q_i)\leq m_{\mathrm{res}}.
\]
\end{proof}

\subsection{Bounds for subquotients at each radius}\label{app:conductor-and-chain-bounds}
We apply Lemmas~\ref{lemA-2} and~\ref{lemA-4} to subquotients
at each radius.
We use the algebras and filtrations of Section~\ref{sec:distribution-algebras}. Let $M$ be a finite nonzero $\Lambda_H$-module
with $d_{\Lambda_H}(M)\leq1$, and put $M_{\an}=D_H\otimes_{\Lambda_H}M$ and
$M_n=D_{H,n}\otimes_{\Lambda_H}M$. Let $n_{\mathrm{res}}$ be the threshold
in~\eqref{eq2-4}. The fixed associated graded module $\gr\overline M$ in that comparison has
Krull dimension at most one. Apply the preceding estimates with $k=k_E$,
$r=d_H$, and $N=\gr\overline M$.

To count successive quotients of canonical dimension one, put
$P=\gr_{r_n}M^\circ$.
Multiply the degrees at radius $r_n$ by $p^n$, as in
Subsection~\ref{subsec:coadmissible-lattices}; then~\eqref{eq2-4}
identifies $P/\eps P$ with the fixed graded module $N$.
Equip a finite chain of submodules of $M_n$ with the induced good filtrations,
and apply Lemma~\ref{lemA-4} to the resulting chain of associated graded
modules. Together with~\eqref{eq2-5}, this gives
\begin{equation}\label{eqA-3}
 \#\{\text{successive quotients of canonical dimension one}\}
 \leq m_{\mathrm{res}}
\end{equation}
for $n\geq n_{\mathrm{res}}$.
The parameter for the fixed associated graded module $\gr\overline M$ is chosen once, so the
bound is independent of the chain and the radius.

\begin{corollary}\label{corA-5}
Suppose $d_{\Lambda_H}(M)\leq1$. Fix $n\geq n_{\mathrm{res}}$. If the
$E$-dimensions of all subquotients of $M_n$ of canonical dimension zero
have a common bound, then $M_n$ has finite length over $D_{H,n}$.
\end{corollary}
\begin{proof}
For this $n$, Lemma~\ref{lemA-4}, applied as in~\eqref{eqA-3}, bounds the
number of successive quotients of canonical dimension one.
Lemma~\ref{lem2-2} shows that subquotients of canonical dimension zero
are finite-dimensional. The successive quotients of canonical dimension one
divide any finite strictly decreasing chain into at most
$m_{\mathrm{res}}+1$ consecutive subchains whose successive quotients
have canonical dimension zero. By exactness, each nonempty subchain
lies in a subquotient of canonical dimension zero. The assumed bound
on its $E$-dimension controls the length of the subchain. This gives a common bound for the
length of the whole chain.
\end{proof}

For the following proposition, assume that $M_{\an}$ has algebraic
infinitesimal character $\chi$, and let $V$ be the algebraic type
determined by $\chi$, as in Lemma~\ref{lem2-4}.
Assume also that the $H_1$-action on $V^\vee$ extends continuously
to $\widehat U_1$, as in Subsection~\ref{subsec:conductors-radii}.

\begin{proposition}\label{propA-3}
There are an integer $n_0=n_0(H,M)\geq n_{\mathrm{res}}$ and a constant
$C_1=C_1(H,M)>0$ such that, for every $n\geq n_0$, every
$D_{H,n}$-subquotient $Q$ of $M_n$ of canonical dimension zero,
written $Q\simeq V^\vee\otimes_EW^\vee$ with $W$ a finite-dimensional smooth
$H$-representation, satisfies
\[
 \dim_EW\leq C_1p^{c_H(W)}.
\]
\end{proposition}
\begin{proof}
Put $d_V=\dim_EV$, and let
$P_i(t)=t^{d_V}+\sum_{b=1}^{d_V}c_{i,b}t^{d_V-b}$ be the characteristic
polynomial of $x_i=\log g_i$ on $V^\vee$.
Choose $n_0\geq n_{\mathrm{res}}$ so that
$v_p(c_{i,b})+b(n_0-w_i)>0$ whenever $c_{i,b}\ne0$.
By Lemma~\ref{lem2-2}, $Q$ is finite-dimensional over $E$. For $Q\ne0$ put
\[
 \nu=\max\{n_0,c_H(W)\},\qquad q=p^\nu\leq p^{n_0}p^{c_H(W)}.
\]
Lemma~\ref{lem2-4} gives $\nu\leq n$ and extends the action on $Q$ to $D_{H,\nu}$.
Consider the evaluation map
\[
 D_{H,\nu}\otimes_{D_{H,n}}Q\longrightarrow Q,\qquad a\otimes s\longmapsto as.
\]
The map $s\mapsto1\otimes s$ has closed image because $Q$ is
finite-dimensional and dense image because $D_{H,n}$ is dense in
$D_{H,\nu}$. Hence the two maps are inverse. By flatness of $D_{H,n}\to D_{H,\nu}$, the displayed isomorphism identifies $Q$ with a subquotient of $M_\nu$.

Put $y_{i,\nu}=p^\nu x_i$. Here $X_i$ denotes the symbol of $g_i-1$
at radius $r_\nu$, with the radius index suppressed as in
Subsection~\ref{subsec:coadmissible-lattices}.
By \cite[Prop. 3.1 and Lem. 4.1]{bib9}, the degree and principal
symbol of $y_{i,\nu}$ are given by
\begin{equation}\label{eqA-1}
 \sigma_{r_\nu}(y_{i,\nu})=X_i^q,\qquad\deg y_{i,\nu}=w_i.
\end{equation}
The smooth factor $W^\vee$ has zero Lie algebra action, so $x_i$ acts on
$Q$ as $x_i|_{V^\vee}\otimes1$.
By Cayley--Hamilton, $Q$ is annihilated by
\[
 p^{\nu d_V}P_i(x_i)=y_{i,\nu}^{d_V}
   +\sum_{b=1}^{d_V}c_{i,b}p^{\nu b}y_{i,\nu}^{d_V-b}.
\]
The inequalities used to choose $n_0$ give the principal symbol
\begin{equation}\label{eqA-2}
 \sigma_{r_\nu}\bigl(p^{\nu d_V}P_i(x_i)\bigr)=X_i^{d_Vq}.
\end{equation}
Choose a $D_{H,\nu}$-submodule $L\subseteq M_\nu$ surjecting onto $Q$.
Give $L$ the filtration induced from $M_\nu$ and $Q$ the resulting quotient
filtration. Thus $\gr L$ is a submodule of
$\gr M_\nu=(\gr_{r_\nu}M^\circ)[1/\eps]$ by~\eqref{eq2-4}.
Inside $\gr M_\nu$, set
\[
 \widetilde C=\gr_{r_\nu}M^\circ\cap\gr L,\qquad
 C=\widetilde C/\eps\widetilde C,\qquad
 I_q=(X_1^{d_Vq},\ldots,X_r^{d_Vq}).
\]
As in the proof of Lemma~\ref{lemA-4},
$\gr_{r_\nu}M^\circ/\widetilde C$ has no $\eps$-torsion.
Hence reduction modulo $\eps$ and~\eqref{eq2-4} give an injection
$C\hookrightarrow\gr\overline M$, while
$\widetilde C[1/\eps]=\gr L$.

Equation~\eqref{eqA-2} gives
\[
 (\widetilde C/I_q\widetilde C)[1/\eps]\twoheadrightarrow\gr Q.
\]
The module $\widetilde C/I_q\widetilde C$ is finite over $k_E[\eps]$,
since $I_q$ contains a power of each $X_i$. The reduction of
$\widetilde C/I_q\widetilde C$ modulo $\eps$ is $C/I_qC$.
Nakayama's lemma at $(\eps)$ bounds the $k_E[\eps]$-rank of
$\widetilde C/I_q\widetilde C$ by $\dim_{k_E}C/I_qC$.
Choose an $E$-basis of $Q$ compatible with its filtration. Then $\gr Q$ has rank $\dim_EQ=d_V\dim_EW$ over $\gr E=k_E[\eps,\eps^{-1}]$. Lemma~\ref{lemA-2}, applied with
exponent $d_Vq$, therefore gives
\[
 \begin{aligned}
 d_V\dim_EW&=\operatorname{rank}_{\gr E}\gr Q\\
 &\leq\operatorname{rank}_{k_E[\eps]}(\widetilde C/I_q\widetilde C)\\
 &\leq\dim_{k_E}C/I_qC\leq C_Nd_Vq.
 \end{aligned}
\]
The bound on $q$ gives the assertion with $C_1=C_Np^{n_0}$.
\end{proof}

\subsection{Proof of the main theorem over \texorpdfstring{$\Q_p$}{Qp}}\label{sec5-6}
We return to $M=\Pi'$ for the representation in
Theorem~\ref{thm:finite-length}, with $F=\Q_p$.
The theorem is immediate if $\Pi=0$, so assume $\Pi\ne0$ throughout
this subsection.

By Lemma~\ref{lem2-3}, $\Pi|_{G_1}$ is admissible and has locally analytic vectors
$\Pi^{\an}$. Choose a sufficiently small compact open uniform subgroup
$J\subseteq G_1$ whose Lie lattice satisfies
$[\log J,\log J]\subseteq p\log J$ for odd $p$, and
$[\log J,\log J]\subseteq8\log J$ for $p=2$. Put $H=J^p$;
then $H$ satisfies~\eqref{eq5-1}.
If $\chi$ is algebraic, choose $J$ small enough that $H$ also satisfies
the condition on the action of $V$ in
Subsection~\ref{subsec:conductors-radii}.
For $L\in\{J,H\}$, define $\nu_L(1)=\infty$ and $\nu_L(g)=a$ for
$g\in L^{p^a}\setminus L^{p^{a+1}}$, where $a\in\Z_{\geq0}$. Use
\[
 \eta=\nu_J+w,\qquad\omega_H=\eta|_H-1=\nu_H+w,\qquad
 w=\begin{cases}1,&p>2,\\3/2,&p=2.\end{cases}
\]
The bracket inclusions give the commutator inequality, and for both
valuations, every element of valuation greater than $p/(p-1)$ has a $p$th
root in the same group.

For this fixed $H$, $\Pi|_H$ is admissible. Its locally analytic vectors have infinitesimal character
$\chi_\Pi|_{Z(U(\mathfrak g_E))}$. For $F=\Q_p$, the nilpotent
cone in $\mathfrak g_E$ has dimension two. Thus
\cite[Thm. 6.1(b) and Lem. 3.9]{bib2}, applied to this representation,
give
\[
 d_{\Lambda_H}(M)=d_{D_H}(M_{\an})<2.
\]
Thus $d_{\Lambda_H}(M)\leq1$ by integrality of canonical dimension.
The subgroup $H$ satisfies the conditions of
Subsection~\ref{subsec:distribution-completions}, and its three basis
weights are $w$. Choose the lattice and presentation of
Subsection~\ref{subsec:coadmissible-lattices}. Equation~\eqref{eq2-4} gives
$\operatorname{Kdim}\gr\overline M\leq1$.
Apply Lemma~\ref{lemA-1} to $N=\gr\overline M$ over $k_E$ and write
$m_{\mathrm{res}}=m_{\mathrm{res}}(\gr\overline M;k',\phi)$ for the rank
defined in Subsection~\ref{app:graded-estimates}.
For every $n\geq n_{\mathrm{res}}$, \eqref{eqA-3} bounds by
$m_{\mathrm{res}}$ the number of successive quotients of canonical dimension
one in any finite chain of $D_{H,n}$-submodules of $M_n$.

If $\chi$ is algebraic, take $n_*=n_0$ and $C_{\mathrm{fd}}=C_1$ from
Proposition~\ref{propA-3}. For every $n\geq n_*$, each $D_{H,n}$-subquotient
$Q=V^\vee\otimes_EW^\vee$ of $M_n$ of canonical dimension zero, with
$W$ smooth, satisfies
\[
 \dim_EW\leq C_{\mathrm{fd}}p^{c_H(W)}.
\]
If $\chi$ is nonalgebraic, take $n_*=n_{\mathrm{res}}$; every
subquotient of canonical dimension zero vanishes.

For this threshold $n_*\geq1$, $M_n$ has finite length over $D_{H,n}$ for every $n\geq n_*$. Indeed, in the algebraic case,
Lemma~\ref{lem2-4} gives $c_H(W)\leq n$, so subquotients of canonical
dimension zero have $E$-dimension at most $d_VC_{\mathrm{fd}}p^n$;
in the nonalgebraic case they vanish. Corollary~\ref{corA-5} gives the conclusion.

\begin{proof}[Proof of Theorem \ref{thm:finite-length} for $F=\Q_p$]
We have $d_{D_H}(M_{\an})\leq1$.
Take a finite strictly decreasing chain of
$G$-stable coadmissible submodules in a subquotient of $M_{\an}$.
For each nonzero successive quotient $S$,
Subsection~\ref{subsec:coadmissible-lattices} gives
\[
 d_{D_{H,n}}(S_n)=d_{D_H}(S)
\]
for all sufficiently large $n$. Since the chain is finite, choose one
$n\geq n_*$ for which this equality holds for every such $S$.
Lift the chain to $M_{\an}$ and apply exact base change.
Equation~\eqref{eqA-3} bounds the number of successive quotients of canonical
dimension one by $m_{\mathrm{res}}$, independently of the radius and the chain.
Propositions~\ref{prop5-5} and~\ref{prop5-6} give a common bound $b_0$
for the lengths of finite strictly decreasing $G$-stable coadmissible
chains in subquotients of canonical dimension zero.
The successive quotients of canonical dimension one divide the chain into at
most $m_{\mathrm{res}}+1$ consecutive subchains. By exactness of canonical
dimension, each nonempty subchain is a chain in a subquotient of canonical
dimension zero, so its length is at most $b_0$. The whole chain therefore
has length at most $m_{\mathrm{res}}+(m_{\mathrm{res}}+1)b_0$.
Lemma~\ref{lem5-8} proves the assertion.
\end{proof}


\begin{thebibliography}{BGR84}
\bibitem[Ard21]{bib10}
K. Ardakov, \emph{Equivariant $\mathcal D$-modules on rigid analytic spaces},
Ast\'erisque \textbf{423} (2021), 1--161,
\url{https://doi.org/10.24033/ast.1145}.

\bibitem[AW13]{bib20}
K. Ardakov and S. Wadsley,
\emph{On irreducible representations of compact $p$-adic analytic groups},
Ann. of Math. (2) \textbf{178} (2013), no.~2, 453--557,
\url{https://doi.org/10.4007/annals.2013.178.2.3}.

\bibitem[BGR84]{bib16}
S. Bosch, U. G\"untzer, and R. Remmert,
\emph{Non-Archimedean Analysis: A Systematic Approach to Rigid Analytic Geometry},
Grundlehren Math. Wiss. \textbf{261}, Springer, Berlin, 1984,
\url{https://doi.org/10.1007/978-3-642-52229-1}.

\bibitem[BS08]{bib1}
M. Boyarchenko and M. Sabitova,
\emph{The orbit method for profinite groups and a $p$-adic analogue of Brown's theorem},
Israel J. Math. \textbf{165} (2008), 67--91,
\url{https://doi.org/10.1007/s11856-008-1004-3}.

\bibitem[Col19]{bib11}
P. Colmez,
\emph{Correspondance de Langlands locale $p$-adique et changement de poids},
J. Eur. Math. Soc. \textbf{21} (2019), no.~3, 797--838,
\url{https://doi.org/10.4171/JEMS/851}.

\bibitem[Con99]{bib15}
B. Conrad, \emph{Irreducible components of rigid spaces},
Ann. Inst. Fourier \textbf{49} (1999), no.~2, 473--541,
\url{https://doi.org/10.5802/aif.1681}.

\bibitem[Dos12]{bib18}
G. Dospinescu,
\emph{Actions infinit\'esimales dans la correspondance de Langlands locale $p$-adique},
Math. Ann. \textbf{354} (2012), no.~2, 627--657,
\url{https://doi.org/10.1007/s00208-011-0736-2}.

\bibitem[DPS23]{bib2}
G. Dospinescu, V. Pa\v{s}k\=unas, and B. Schraen,
\emph{Gelfand--Kirillov dimension and the $p$-adic Jacquet--Langlands correspondence},
J. Reine Angew. Math. \textbf{801} (2023), 57--114,
\url{https://doi.org/10.1515/crelle-2023-0033}.

\bibitem[Eme17]{bib12}
M. Emerton,
\emph{Locally analytic vectors in representations of locally $p$-adic analytic groups},
Mem. Amer. Math. Soc. \textbf{248} (2017), no.~1175,
\url{https://doi.org/10.1090/memo/1175}.

\bibitem[HW24]{bib3}
Y. Hu and H. Wang,
\emph{On some $p$-adic and mod $p$ representations of quaternion algebra over $\Q_p$},
J. Reine Angew. Math. \textbf{812} (2024), 163--210,
\url{https://doi.org/10.1515/crelle-2024-0025}.

\bibitem[Laz65]{bib4}
M. Lazard, \emph{Groupes analytiques $p$-adiques},
Publ. Math. Inst. Hautes \'Etudes Sci. \textbf{26} (1965), 5--219,
\url{https://doi.org/10.1007/BF02684303}.

\bibitem[LW25]{bib5}
H. Liu and H. Wang,
\emph{On the finite length of some $p$-adic representations of the quaternion algebra over $\Q_p$},
arXiv:2510.19218v1, October 22, 2025,
\url{https://arxiv.org/abs/2510.19218v1}.

\bibitem[Pa\v{s}22]{bib17}
V. Pa\v{s}k\=unas, \emph{On some consequences of a theorem of J. Ludwig},
J. Inst. Math. Jussieu \textbf{21} (2022), no.~3, 1067--1106,
\url{https://doi.org/10.1017/S1474748020000547}.

\bibitem[ST01]{bib21}
P. Schneider and J. Teitelbaum, with an appendix by D. Prasad,
\emph{$U(\mathfrak g)$-finite locally analytic representations},
Represent. Theory \textbf{5} (2001), 111--128,
\url{https://doi.org/10.1090/S1088-4165-01-00109-1}.

\bibitem[ST02]{bib7}
P. Schneider and J. Teitelbaum,
\emph{Locally analytic distributions and $p$-adic representation theory, with applications to $\GL_2$},
J. Amer. Math. Soc. \textbf{15} (2002), 443--468,
\url{https://doi.org/10.1090/S0894-0347-01-00377-0}.

\bibitem[ST03]{bib6}
P. Schneider and J. Teitelbaum,
\emph{Algebras of $p$-adic distributions and admissible representations},
Invent. Math. \textbf{153} (2003), 145--196,
\url{https://doi.org/10.1007/s00222-002-0284-1}.

\bibitem[Sch18]{bib8}
P. Scholze, \emph{On the $p$-adic cohomology of the Lubin--Tate tower},
Ann. Sci. \'Ec. Norm. Sup\'er. (4) \textbf{51} (2018), no.~4, 811--863,
\url{https://doi.org/10.24033/asens.2367}.

\bibitem[Sor26]{bib9}
R. Sorgdrager,
\emph{Gelfand--Kirillov bound for $p$-adic Banach representations with infinitesimal character for $\GL_2$ and quaternion units},
arXiv:2602.08856v1, February 9, 2026,
\url{https://arxiv.org/abs/2602.08856v1}.

\bibitem[Stacks]{bib14}
The Stacks Project Authors, \emph{The Stacks Project},
Sections 10.56 (\href{https://stacks.math.columbia.edu/tag/00JL}{tag 00JL})
and 10.58 (\href{https://stacks.math.columbia.edu/tag/00JV}{tag 00JV}),
and Lemmas 10.63.5 (\href{https://stacks.math.columbia.edu/tag/00LC}{tag 00LC}),
10.104.2 (\href{https://stacks.math.columbia.edu/tag/02JN}{tag 02JN}),
10.112.4 (\href{https://stacks.math.columbia.edu/tag/00OK}{tag 00OK}),
and 10.144.2 (\href{https://stacks.math.columbia.edu/tag/00UC}{tag 00UC}).

\bibitem[Ven02]{bib19}
O. Venjakob,
\emph{On the structure theory of the Iwasawa algebra of a $p$-adic Lie group},
J. Eur. Math. Soc. \textbf{4} (2002), no.~3, 271--311,
\url{https://doi.org/10.1007/s100970100038}.
\end{thebibliography}
\end{document}